\documentclass[svgnames]{amsart}
\usepackage[english]{babel}
\usepackage[utf8]{inputenc}

\usepackage{svg}
\usepackage{hyperref}
\usepackage{enumitem}
\hypersetup{colorlinks,citecolor=Black,linkcolor=Black,urlcolor=Blue}
\usepackage{amssymb, amsfonts, amsmath, amsthm, amscd, amsxtra, latexsym, mathabx,mathtools}

\usepackage{tikz-cd}

\theoremstyle{plain}
\newtheorem{thm}{Theorem}[section] 
\newtheorem{lemma}[thm]{Lemma}
\newtheorem{cor}[thm]{Corollary}
\newtheorem{prop}[thm]{Proposition}
\newtheorem{claim}[thm]{Claim}

\newtheorem{thmintro}{Theorem}

\newtheorem{conjintro}[thmintro]{Conjecture}
\newtheorem{corintro}[thmintro]{Corollary}

\theoremstyle{definition}
\newtheorem{defn}[thm]{Definition}
\newtheorem{rem}[thm]{Remark}
\newtheorem{notation}[thm]{Notation}

\newtheorem{question}[thmintro]{Question}
\newtheorem{ex}[thm]{Example}

\newenvironment{claimproof}{\begin{proof}}{\end{proof}}

\DeclarePairedDelimiter{\abs}{\lvert}{\rvert}

\newcommand{\aut}[1]{\operatorname{Aut}(#1)}
\newcommand{\out}[1]{\operatorname{Out}(#1)}

\newcommand{\Z}{\mathbb{Z}}

\newcommand{\BC}{\mathcal{BC}}

\newcommand{\centre}[2]{\operatorname{C}_{#1}(#2)}
\newcommand{\crush}[2]{\mathcal{#1}^{#2}}

\newcommand{\edge}[1]{\operatorname{E}(#1)}
\renewcommand{\epsilon}{\varepsilon}

\DeclareMathOperator*{\free}{\ast}

\DeclarePairedDelimiter{\nspan}{\langle\mkern-5mu\langle}{\rangle\mkern-5mu\rangle}

\newcommand{\Stab}[2]{\operatorname{Stab}_{#1}\left(#2\right)}
\newcommand{\dist}{\operatorname{d}}

\renewcommand{\phi}{\varphi}

\newcommand{\ribb}[3]{\operatorname{Ribb}_{#1}(#2,#3)}

\newcommand{\ver}[1]{\operatorname{V}(#1)}

\title[A combination theorem for the twist conjecture for Artin groups]{A combination theorem for the twist conjecture for Artin groups}

\author[O. Jones]{Oli Jones}
    \address{(Oli Jones) Institute of Mathematics\\ Technische Universität Berlin\\ Berlin, Germany}
    \email{jones@math.tu-berlin.de} 

\author[G. Mangioni]{Giorgio Mangioni}
    \address{(Giorgio Mangioni) Maxwell Institute and Department of Mathematics\\ Heriot-Watt University\\ Edinburgh, UK\\ Orcid: 0000-0003-2868-5032}
    \email{gm2070@hw.ac.uk}

\author[G. Sartori]{Giovanni Sartori}
    \address{(Giovanni Sartori) Maxwell Institute and Department of Mathematics\\ Heriot-Watt University\\ Edinburgh, UK}
    \email{gs2057@hw.ac.uk}

\begin{document}

\begin{abstract}
    We reduce a strong version of the twist conjecture for Artin groups to Artin groups whose defining graphs have no separating vertices. This produces new examples of Artin groups satisfying the conjecture, and sheds more light on the isomorphism problem for Artin groups. Along the way we also prove a combination result for the ribbon property for vertices.
\end{abstract}

\maketitle

\small



\normalsize

\section*{Introduction}
Artin groups are a rich class of groups generalising braid groups and with strong connections to Coxeter groups. They are defined via the following presentation: given a finite simplicial graph $\Gamma$ with vertices $\ver\Gamma$, edges $\edge\Gamma$, and for each edge $\{a,b\} \in \edge{\Gamma}$ a label $m_{ab} \in \mathbb{N}_{\geq2}$, the associated \emph{Artin group} $A_\Gamma$ is 
$$\langle \ver{\Gamma} \ | \ \mathrm{prod}(a,b,m_{ab})=\mathrm{prod}(b,a,m_{ab}) \, \forall \{a,b\}\in \edge{\Gamma}\rangle,$$
where $\mathrm{prod}(u,v,n)$ denotes the prefix of length $n$ of the infinite alternating word $uvuvuv\cdots$. When $a$ and $b$ do not span an edge, we abuse notation and set $m_{ab}=\infty$. 

One of the most prominent open questions for Artin groups is the \emph{isomorphism problem}, which specialises Dehn's isomorphism problem~\cite{Dehn_isoprob}, and asks for an algorithm that, given two labelled graphs, determines if they give rise to isomorphic Artin groups. The question is completely open in general, with some results when one or both graphs are restricted to some subclasses {\cite{baudisch1981raagcharacterise,droms1987isomorphisms,Paris_isoprob_spherical,vaskou_isoproblem}}. In this paper we are concerned with the twist conjecture, which can be seen as a first step towards the isomorphism problem. Before giving more details, we need some terminology.

An \emph{Artin system} is a pair $(A,S)$ where $A$ is a group isomorphic to some $A_\Gamma$, and the isomorphism maps $S\subseteq A$ bijectively onto the vertices of $\Gamma$ (see Definition~\ref{defn:artin}). The set $S$ is called an \emph{Artin generating set} for $A$, and the graph $\Gamma$, which we also denote by $\Gamma_S$, is the corresponding \emph{defining graph} or \emph{Coxeter graph}. A \emph{standard parabolic subgroup} of $(A,S)$ is the subgroup $A_Y$ generated by some $Y\subseteq S$; by a result of van der Lek, $(A_Y,Y)$ is itself an Artin system \cite{VanDerLek}.

Two Artin generating sets for the same Artin group are \emph{twist equivalent} if they are related by a sequence of \emph{elementary twists}, whose definition we briefly recall, postponing all details to Definition~\ref{defn:twist_eq}. Given an Artin generating set $S$, let $Y\subseteq S$ be a subset which separates $\Gamma_S$, in the sense of Definition~\ref{defn:separating}. Suppose that the Artin system $(A_Y,Y)$ is of spherical type (meaning the corresponding Coxeter group is finite) and does not decompose as a direct product. Performing an elementary twist of $S$ along $Y$ roughly amounts to conjugating one of the connected components of $\Gamma_S-\Gamma_Y$ by the \emph{Garside element} associated to $(A_Y,Y)$ (this is a distinguished element corresponding to the longest element in the associated Coxeter group, see \cite{BrieskornSaito}). Two graphs $\Gamma$ and $\Gamma'$ are said to be \emph{twist equivalent} if $\Gamma\cong \Gamma_S$ and $\Gamma'\cong\Gamma_U$ for some twist equivalent Artin generating sets $S$ and $U$ of an Artin group $A$.

 In all the solved cases of the isomorphism problem, two graphs give rise to isomorphic Artin groups if and only if they are twist equivalent, so it is natural to ask if this is always the case:

\begin{question}[{\cite[{Problem~28}]{charney2016problems}}]
    If two labelled graphs $\Gamma$ and $\Delta$ give rise to isomorphic Artin groups, are they necessarily twist equivalent?
\end{question}

As a first step in this direction, Brady, McCammond, M\"uhlherr, and Neumann conjectured the following:
\begin{conjintro}[{Weak twist conjecture, \cite[Conjecture 8.2]{BMMNtwistconjecture}}]\label{conjintro_WTC}
Let $(A, S)$ be an Artin system. If $U \subseteq A$ is an Artin generating set with $R_S=R_U$, then $\Gamma_S$ and $\Gamma_U$ are twist equivalent. 
\end{conjintro}
Here $R_S$, called the \emph{reflection set} of $(A,S)$, is the union of the conjugacy classes of $S$ in $A$, and similarly for $R_U$. The corresponding conjecture for Coxeter groups, which is \cite[Conjecture 8.1]{BMMNtwistconjecture}, has been proven to be false in \cite{RT_Cox_not_twisteq}; however, Artin groups are believed to be more rigid than their Coxeter counterparts. 

In this paper we consider the following stronger version of the twist conjecture:
\begin{conjintro}[Strong twist conjecture]\label{conjintro:Stc}
Let $(A, S)$ be an Artin system. If $U \subseteq A$ is an Artin generating set with $R_S=R_U$, then $S$ and $U$ are twist equivalent. 
\end{conjintro}
The analogue of the strong twist conjecture for Coxeter groups (with the appropriate extra conditions to avoid the known counter examples) was already considered by M\"uhlherr in \cite[Conjecture 2]{Muhlherr}.

\subsection*{A combination theorem for the strong twist conjecture.}
A \emph{1--chunk}\footnote{Such a subgraph has appeared in the literature under the name of \emph{biconnected component}, but we prefer the term 1--chunk both to mimic the terminology \emph{big chunk} that we used in \cite{JonManSar_JSJSpl}, and because the word \emph{chunk} has already been used in e.g. \cite{crisp2005automorphisms,an2022automorphismgroupsartingroup,BMV,Jones_VF} to denote maximal subgraphs without separating edges or vertices.} of a simplicial graph is a connected induced subgraph without separating vertices which is maximal with these properties. Given an Artin system $(A,S)$, a \emph{1--chunk parabolic subgroup} is a subgroup conjugated to some $A_X$, where $X\subseteq S$ spans a 1--chunk in $\Gamma_S$. In our previous work, we proved that every isomorphism between two Artin systems induces a bijection between the sets of 1--chunks, which preserves the isomorphism type~\cite{JonManSar_JSJSpl}. In light of this, one might hope to recover the strong twist conjecture for an Artin system from the twist conjecture for its 1--chunk parabolic subgroups. We prove just that in our first main result, under a further mild hypothesis:
\begin{thmintro}[see Theorem~\ref{thm:strong+ribbon=strong}]\label{thmintro:strong+ribbon=strong}
  Let $(A,S)$ be an Artin system, with $\Gamma_S$ connected. Suppose that the following holds:
    \begin{itemize}
        \item For every $X\subseteq S$ spanning a 1--chunk, $(A_X,X)$ satisfies the strong twist conjecture;
        \item For every $Y\subseteq S$ spanning a clique, $(A_Y,Y)$ satisfies the vertex ribbon property.
    \end{itemize} 
    Then $(A,S)$ satisfies the strong twist conjecture.
\end{thmintro}
First introduced by Paris~\cite{Pparabolics}, a \emph{ribbon} can be intuitively thought of as a minimal element that conjugates two standard parabolic subgroups of an Artin system (see Definition~\ref{def:ribbon}, which for our purposes only describes ribbons between standard generators). The \emph{vertex ribbon property} states that any element that conjugates two standard generators decomposes as a product of ribbons, and is conjectured to hold for all Artin systems. In the course of proving Theorem~\ref{thmintro:strong+ribbon=strong}, we also obtain a combination result for the vertex ribbon property, see Theorem~\ref{thmintro:ribbon_extend} below.

As an application of Theorem~\ref{thmintro:strong+ribbon=strong}, we exhibit new families of Artin systems satisfying the strong twist conjecture. Our result covers Artin systems for which even the weak twist conjecture was not previously known (see Figure~\ref{fig:strongtwist_new_example} for an example).
\begin{corintro}[see Corollary~\ref{cor:strong_twist_for_strong_chunks}]\label{corintro:strong_twist_for_strong_chunks}
Let $(A,S)$ be an Artin system, with $\Gamma_S$ connected. Suppose that, for every $X\subseteq S$ spanning a 1--chunk, the Artin system $(A_X,X)$ is of one of the following types:
\begin{enumerate}
\item\label{1} $(A_X,X)$ is right-angled (i.e. all edge labels are $2$);
\item\label{2} $(A_X,X)$ is of large-type (i.e. all edge labels are at least $3$), triangle-free (i.e. no induced subgraph of $\Gamma_X$ is a triangle) and has no separating edge;
\item\label{4} $(A_X,X)$ is of type XXXL (i.e. all edge labels are at least $6$) and has no separating edge;
\item\label{3} $(A_X,X)$ is of large-type and free-of-infinity (i.e. $\Gamma_X$ is a complete graph);
\item\label{5} $(A_X,X)$ is of type $A_n$ for $n \geq 2$; $B_n$ for $n \geq 2$; or $D_n$ for $n \geq 4$ and $n \neq 5$.
\end{enumerate}
Then $(A,S)$ satisfies the strong twist conjecture.
\end{corintro}

In Lemmas \ref{lem:strong_RAAG}-\ref{lem:strong_large}-\ref{lem:strong_spherical}, by combining existing results in the literature, we show that the classes of Artin systems appearing as 1--chunks above satisfy the strong twist conjecture. Furthermore, if $(A,S)$ is as in Corollary~\ref{corintro:strong_twist_for_strong_chunks}, the Artin systems supported on cliques of $\Gamma_S$ enjoy the vertex ribbon property, and actually the more general ribbon property. Indeed, the ribbon property in case \eqref{1} is work of Godelle~\cite{godelle2003parabolic}; case \eqref{5} is covered by results of Paris and Godelle for spherical type Artin groups \cite{Pparabolics, godelle2002normalisateurs}; finally, Artin systems of type \eqref{2}-\eqref{3}-\eqref{4} are two-dimensional, and the ribbon property follows from \cite{godelle2007artin}.

\begin{figure}[htp]
\centering
\includegraphics[width=\textwidth, alt={An example of an Artin group whose 1--chunks are as in the theorem, which therefore satisfies the strong twist conjecture}]{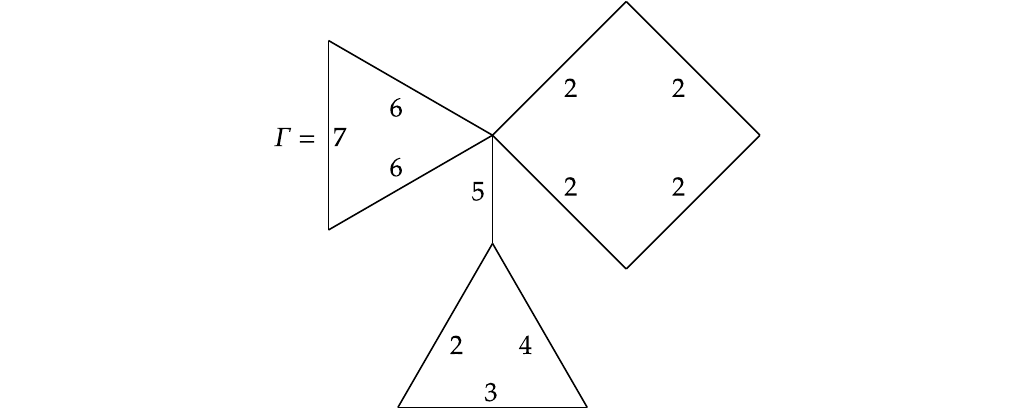}
\caption{In the labelled graph $\Gamma$ above, the 1--chunks are the two triangles, the square, and the edge labelled by $5$. They are all of the type described in Corollary~\ref{corintro:strong_twist_for_strong_chunks}, so $A_\Gamma$ satisfies the strong twist conjecture.}
\label{fig:strongtwist_new_example}
\end{figure}

Theorem \ref{thmintro:strong+ribbon=strong} and Corollary \ref{corintro:strong_twist_for_strong_chunks} should be compared to \cite{ratcliffe2013jsj}, where the authors describe a JSJ decomposition of Coxeter groups over FA subgroups, and produce a similar combination theorem for the (strong) twist conjecture, allowing them to conclude that all Coxeter systems with chordal defining graphs satisfy the strong twist conjecture. Indeed, underlying our techniques is an explicit JSJ decomposition for any Artin group, which we constructed in our previous work \cite{JonManSar_JSJSpl}.

If we drop the assumption on the ribbon property in Theorem~\ref{thmintro:strong+ribbon=strong}, we still get a sufficient condition for the weak twist conjecture:

\begin{thmintro}[see Theorem~\ref{thm:refrigidreduction}]\label{thmintro:refrigidreduction}
    Let $(A,S)$ be an Artin system, with $\Gamma_S$ connected. Suppose that, for every $X\subseteq S$ spanning a 1--chunk, $(A_X,X)$ satisfies the strong twist conjecture. Then $(A,S)$ satisfies the weak twist conjecture.
\end{thmintro}
In the above setup, we actually show that every Artin generating set $U$ with $R_S=R_U$ is related to $S$ by a finite sequence of elementary twists and \emph{Dehn twists} (see Definition~\ref{defn:dehn_twist}). Indeed, Theorem~\ref{thmintro:strong+ribbon=strong} follows from Theorem~\ref{thmintro:refrigidreduction} by using the vertex ribbon property to describe Dehn twists as compositions of elementary twists.

\subsection*{A combination theorem for the vertex ribbon property}
As a collateral result of the proof of Theorem~\ref{thmintro:strong+ribbon=strong}, we can reduce the vertex ribbon property to the free-of-infinity Artin groups.

\begin{thmintro}[see Corollary~\ref{cor:ribbon_for_cliques_ribbon_for_all}]\label{thmintro:ribbon_extend} 
    Let $(A,S)$ be an Artin system. Suppose that, for any $Y \subseteq S$ spanning a clique in $\Gamma_S$, $(A_Y,Y)$ satisfies the vertex ribbon property. Then $(A,S)$ satisfies the vertex ribbon property. 
\end{thmintro}
The above result is a consequence of a more general combination theorem, stating that if an Artin system $(A,S)$ admits a visual splitting and the factors satisfy the vertex ribbon property then so does the whole $(A,S)$ (see Proposition~\ref{prop:combination_ribbons}).

\subsection*{Organisation of the paper}
Section~\ref{sec:background} contains generalities on simplicial actions on trees, Bass--Serre theory and \emph{deformation spaces}, in the sense of Forester \cite{forester}, which are the main technical tools of this paper. In Section~\ref{sec:cyclicsplit} we recall some properties of Artin groups; then, building on results from our previous paper \cite{JonManSar_JSJSpl}, we associate a deformation space to every Artin system, whose trees roughly correspond to the maximal visual splittings over 1--chunk parabolic subgroups.

In Section~\ref{sec:red twist conj}, we prove a technical intermediate theorem, involving a generalised version of the strong twist conjecture that allows Dehn twists (as well as elementary twists and conjugations), which in particular implies Theorem \ref{thmintro:refrigidreduction}. In Section~\ref{sec:strongtwist_for_some}, we upgrade this theorem under the further mild assumption that the 1--chunks satisfy the ribbon property. In particular, we obtain Theorem~\ref{thmintro:strong+ribbon=strong}, which we then combine with the strong twist conjecture for several classes of Artin groups to get Corollary~\ref{corintro:strong_twist_for_strong_chunks} (see Corollaries~\ref{thm:strong+ribbon=strong} and~\ref{cor:strong_twist_for_strong_chunks}, respectively). In the process we also prove the combination result for the ribbon property, which then yields Theorem~\ref{thmintro:ribbon_extend} (see Corollary~\ref{cor:ribbon_for_cliques_ribbon_for_all}).

\subsection*{Acknowledgements}
We are grateful to Piotr Przytycki, María Cumplido, and Alexandre Martin for insights about the twist conjecture and the ribbon property, and to Stefanie Zbinden for helpful conversations. A special thanks goes to Laura Ciobanu and Alessandro Sisto for comments on a first draft of this document. Finally, we would like to thank the anonymous referee for their careful reading of the manuscript and helpful suggestions.

\section{Background on groups acting on trees}\label{sec:background}
\noindent We recall here some properties of simplicial actions on trees. We shall work in the following setting:

\begin{notation}\label{notation:action_on_trees}
    By \emph{tree} we mean a simply connected simplicial graph, equipped with the metric where each edge has length one. Given a group $G$, a \emph{$G$-tree} $(T,\Omega)$ is a tree $T$ endowed with an action $\Omega\colon G\to \aut{T}$ by simplicial isometries, without edge inversions, and \emph{minimal} (i.e. no proper sub-tree is invariant under the action). We often suppress the reference to the action $\Omega$ when it is not relevant or it is clear from the context. Similarly, for every $g\in G$ and $X\subseteq T$ we often denote $\Omega(g)(X)$ by $g\cdot X$ or simply $gX$. 
\end{notation}
\begin{rem}
    Bass \cite[Proposition~7.9]{Bass} proved that, if $G$ is finitely generated, a minimal $G$-action on a tree $T$ is also cocompact.
\end{rem}

\noindent Throughout we will write $\Stab{\Omega}{S}$ for the pointwise stabiliser of $S \subseteq T$, or $\Stab{G}{S}$ if there is no danger of confusing the action.

 The \emph{translation length} of an element $g\in G$ is defined as $|g|\coloneq\inf_{x\in T}\dist_T(x,g\cdot x)$. The \emph{minset} of $g$ is the subtree spanned by all points $x$ which realise the translation length.  If $|g|=0$ the element is called \emph{elliptic}, and its minset is the sub-tree of all fixed points of $g$. If otherwise $|g|>0$ the element is \emph{loxodromic}, and its minset is a geodesic line on which $g$ acts by translations. A subgroup $H\le G$ acts \emph{elliptically} on $T$ (or, is an elliptic subgroup) if it fixes a vertex of $T$.

\subsection{Bass--Serre theory} We recall the concept of a graph of groups and show that it is equivalent to the simplicial action of a group on a tree. We refer to \cite{trees} for the definitions and results in this subsection. As we do not need the full generality of Bass--Serre theory, we adapt the definitions and results to the setting of graphs of groups based on a tree. 

\begin{defn}[Graph of groups]
    Let $Y$ be a simplicial tree;
    \begin{enumerate}
        \item a \emph{graph of groups} on $Y$ is a system $\mathcal G = \{G_v\}_{v\in\ver Y}\cup\{G_e\}_{e\in\edge Y}$ of groups and group monomorphisms $\psi_e^v\colon G_e\to G_{v}$, whenever $v$ is a vertex of $e$. The groups $G_v$'s are the \emph{vertex groups} of $\mathcal G$, whereas the groups $G_e$'s are the \emph{edge groups} of $\mathcal G$;
        \item the \emph{fundamental group $\pi_1(\mathcal G,Y)$ of $\mathcal G$} is the colimit of $\mathcal G$, i.e. the group 
        \[
        \frac{\free_{v\in\ver Y}G_v}{\nspan{\psi_{\{a,b\}}^a(g)\psi_{\{a,b\}}^b(g)^{-1}:\{a,b\}\in\edge Y, g\in G_{ab}}}.
        \]
    \end{enumerate}
\end{defn}
\noindent Note that, more in general, graph of groups can be defined on graphs that are not simplicial trees. In that case, the underlying graph needs to be an oriented graph, possibly with loops.  Given a group $G$, a \emph{graph-of-groups decomposition} of $G$ is a graph of groups $\mathcal G$ on some tree $Y$ such that $G\cong \pi_1(\mathcal G,Y)$. 

\begin{ex}[Amalgamated products]
    \label{ex:amalgProd}
    Each amalgamated product of groups $G=A*_CB$ admits a graph-of-groups decomposition as follows: let $Y$ consist of a single edge $\{a,b\}$ and let $\mathcal G$ be defined as $G_a=A$, $G_b=B$ and $G_{\{a,b\}}=C$, with the monomorphisms $G_{\{a,b\}}\to G_a$ and $G_{\{a,b\}}\to G_b$ being the inclusion maps. The universal property of amalgamated products gives $\pi_1(\mathcal G,Y)\cong G$.
\end{ex}

\begin{defn}[Bass--Serre tree]
    Let $\mathcal G$ be a graph of groups on an oriented graph~$Y$; the \emph{Bass--Serre tree}\footnote{This tree is sometimes called the \emph{universal covering} of $\mathcal G$.} of $\mathcal G$ is the $\pi_1(\mathcal G,Y)$-graph $\tilde X(\mathcal G)$ defined as follows. The set of vertices of $\tilde X(\mathcal G)$ is $\bigsqcup\{\pi_1(\mathcal G,Y)/G_v:v\in\ver Y\}$, while its set of edges is $\bigsqcup\{\pi_1(\mathcal G,Y)/G_e:e\in\edge Y\}$ (the edge $gG_{\{a,b\}}$ connecting the vertices $gG_a$ and $gG_b$).\par
    
    The fundamental group $\pi_1(\mathcal G,Y)$ acts on the vertices of $\tilde X(\mathcal G)$ by left multiplication and this action descends to a simplicial action without inversion on $\tilde X(\mathcal G)$.
\end{defn}

\noindent Notice that, for every $g\in \mathcal G$ and $v\in \ver Y$, the stabiliser of $gG_v$ in $\mathcal G$ is $gG_vg^{-1}$, which we also denote by $(G_v)^g$. Similarly, for every $e\in \edge Y$, the stabiliser of $gG_e$ in $\mathcal G$ is $(G_e)^g\coloneq gG_eg^{-1}$.

The term Bass--Serre \emph{tree} is justified by the following theorem:

\begin{thm}
    Let $\mathcal G$ be a graph of groups over a tree $Y$; then the graph $\tilde X(\mathcal G)$ is a simplicial tree and the quotient graph $\tilde X(\mathcal G)/\pi_1(\mathcal G,Y)$ is isomorphic to $Y$. \par 
    Moreover, let $X\subseteq \tilde X(\mathcal G)$ be any connected fundamental domain for the action of $\pi_1(\mathcal G,Y)$ on $\tilde X(\mathcal G)$ and let $\mathcal H$ be the graph of groups on $X$ defined in the following way: for every $v\in\ver X$, let $H_v=\Stab{\pi_1(\mathcal G,Y)}{v}$; for every $e\in\edge X$, let $H_e=\Stab{\pi_1(\mathcal G,Y)}{e}$; if $v$ is an endpoint of $e$, then let the monomorphism $\iota_e^v\colon H_e\to H_v$ be the inclusion map. The graphs of groups $\mathcal G$ and $\mathcal H$ are isomorphic (as systems of groups).
\end{thm}

\noindent A converse result also holds:

\begin{thm}
    Let $M$ be a simplicial tree, let $G$ be a group acting on $M$ without inversion. Moreover, assume that $M/G$ is a tree and let $X\subseteq M$ be a connected fundamental domain with respect to this action. Let $\mathcal H$ be the graph of groups on~$X$ constructed by taking the system of stabilisers (as in the previous theorem). Then $G\cong\pi_1(\mathcal H,X)$ and there is an equivariant isomorphism between the $G$-tree $M$ and the $\pi_1(\mathcal H,X)$-tree $\tilde X(\mathcal H)$.
\end{thm}

\noindent In light of this equivalence, we will often pass freely between a graph of groups and its Bass--Serre tree.

\begin{defn}\label{def_splitoverZ}
    A \emph{splitting over $\Z$} of a group $G$ is a $G$-tree $T$ with infinite cyclic edge stabilisers. We then say that $G$ \emph{splits over $\Z$}.
\end{defn}
\subsection{Deformation spaces}
For the remainder of this section we follow \cite{GL_def_trees}, and recall some properties of \emph{deformation spaces}, a notion originally due to Forester \cite{forester}. Intuitively, a deformation space is a space parametrised by $G$-trees.

\begin{defn}\label{def:elementarycollapse}
    We say an edge $e=\{u,v\}$ of a $G$-tree $(T,\Omega)$ is \emph{collapsible} if $\Stab{G}{u} \le \Stab{G}{v}$, and $u,v$ lie in different $G$-orbits. The associated \emph{elementary collapse} produces a new $G$-tree $(T', \Omega')$ as follows: for every $g\in G$ one removes $g\cdot e$ from $T$ and identifies $g\cdot u$ and $g\cdot v$. This gives a $G$-equivariant map $\phi\colon T\to T'$. The inverse of an elementary collapse is called an \emph{elementary expansion}. An \emph{elementary deformation} is a finite sequence of elementary collapses and expansions. 
\end{defn}

\noindent Elementary collapses preserve elliptic subgroups, in the following sense:
\begin{lemma}
    \label{lem:stabsOfCollapse}
    Let $(T,\Omega)$ be a $G$-tree, let $e=\{u,v\}$ be an edge of $T$ with $u,v$ in different $\Omega$-orbits and $\Stab \Omega u\le\Stab\Omega v$, and let $\phi\colon(T,\Omega)\to(T',\Omega')$ be the elementary collapse of the edge~$e$. Every edge stabiliser (resp. vertex stabiliser) for $\Omega'$ is also an edge stabiliser (resp. vertex stabiliser) for $\Omega$.
     In particular:
    \begin{enumerate}
        \item if $x'$ is an edge of $T'$ (resp. a vertex of $T'$ not in the orbit of $\phi(e)$), then for every edge (resp. vertex) $x$ of $T$ such that $\phi(x) = x'$, $\Stab\Omega x=\Stab{\Omega'}{x'}$;
        \item $\Stab{\Omega'}{\phi(e)}=\Stab\Omega v$.
    \end{enumerate}
\end{lemma}

\begin{proof}
We first notice that, if $x'$ is an open edge of $T'$ (note that restricting to the interior of edges does not change the stabilisers), or a vertex which is not in the $G$-orbit of $\phi(e)$, then $\phi$ is injective on $\phi^{-1}(x')$. Now suppose $\phi(x) = x'$. Then if $gx' = x'$ it follows by $G$-equivariance that $\phi(x)= x'= gx' = \phi(gx)$, so $x = gx$ by injectivity. Conversely if $gx = x$ then $gx' = g\phi(x) = \phi(gx) = x'$ by equivariance. Hence, we have that $\Stab{\Omega'}{x'}=\Stab{\Omega}{x}$ for any edge (resp. vertex) $x$ such that $\phi(x) = x'$.

To conclude the proof, we now show that
\begin{equation}\label{eq:stab_phi(e)}
    \Stab{\Omega'}{\phi(e)}=\Stab{\Omega}{v}.
\end{equation}
    For the inclusion $\le$, let $g\in \Stab{\Omega'}{\phi(e)}$. Towards a contradiction, assume that $g\cdot v\neq v$. Because $g$ fixes $\phi(e)$, the vertices $v$ and $g\cdot v$ have the same image under~$\phi$. Then there is a path of $G$-translates $e_1=g_1\cdot e, e_2=g_2\cdot e_1,\ldots, e_k=g_k \cdot e_{k-1}$ of $e$ connecting $v$ to $g\cdot v$. Notice that $k\ge 2$, as $v$ and $u$ are in different $G$-orbits; moreover $g_1\cdot v=v$, as the path connects $v$ to $g\cdot v$, and for the same reason $g_2$ fixes $g_1 \cdot u$ but not $v$. However $g_2 \in \Stab{\Omega}{g_1 \cdot u}=\Stab{\Omega}{u}^{g_1}\le \Stab{\Omega}{v}$, which is absurd.

    For the converse inclusion, if $g\in \Stab{\Omega}{v}$ then $$g\phi(e)=g\phi(v)=\phi(g\cdot v)=\phi(v)=\phi(e),$$
    where we used that $\phi$ is $G$-equivariant.
\end{proof}
\begin{defn}
    Two $G$-trees $(T,\Omega)$ and $(T',\Omega')$ are \emph{$G$-equivariantly isometric} if there is a simplicial isometry $f\colon T\to T'$ such that, for every $g\in G$, we have that $f\circ \Omega(g)=\Omega'(g) \circ f$.
\end{defn}
\begin{defn}\label{defn:deformation space}
Given a $G$-tree $(T, \Omega)$, the \emph{deformation space} containing $(T, \Omega)$ is the simplicial realisation of the following partial order. Take the underlying set to be all $G$-trees related to $(T,\Omega)$ by an elementary deformation, up to $G$-equivariant isometry. Then say $(T', \Omega') \geq (T'', \Omega'')$ if $(T', \Omega')$ admits a sequence of elementary collapses to $(T'', \Omega'')$.
\end{defn}

\begin{rem}
For experts, what we refer to here as a deformation space may more accurately be called the simplicial spine. This is a deformation retract of the genuine deformation space, which is obtained by considering a space of metric trees (instead of only simplicial trees). We restrict to the spine for the ease of exposition.
\end{rem}

\noindent The following theorem of Forester is a very convenient characterisation of when two $G$-trees are in the same deformation space:

\begin{thm}\cite[Theorem 1.1]{forester}\label{thm:forester}
    Let $(T, \Omega)$ and $(T', \Omega')$ be cocompact $G$-trees. Then the following are equivalent:
    \begin{enumerate}
        \item $(T, \Omega)$ and $(T', \Omega')$ are related by an elementary deformation (i.e. they belong to the same deformation space).
        \item $(T, \Omega)$ and $(T', \Omega')$ have the same elliptic subgroups.
    \end{enumerate}
\end{thm}

\begin{defn}
    A $G$-tree is \emph{reduced} if no elementary collapse is possible.
\end{defn}

\begin{rem}\label{rem:reduced JSJ tree}
    If a $G$-tree $T$ is not reduced, one obtains a (possibly non-unique) reduced $G$-tree $T'$ by collapsing collapsible edges until no collapse is possible. This procedure eventually ends, as $T/G$ is finite.
\end{rem}

\begin{defn}
    Given a $G$-tree $T$, we say (the orbit of) an edge $e$ is \emph{surviving} if there exists a sequence of elementary collapses from $T$ to a reduced tree, such that $e$ is not collapsed. We say $T$ is surviving if every orbit of edges is surviving.
\end{defn}

\noindent A particularly well-behaved subclass of deformation spaces are the \emph{non-ascending} ones. We first need a definition.

\begin{defn}[{see \cite[Section 1.2]{GuirardelLevitt}}]
    A deformation space $\mathcal D$ is \emph{irreducible} if there exists $(T,\Omega)\in \mathcal D$ and two loxodromic elements for the action $\Omega$ whose commutator is again loxodromic.
\end{defn}

\begin{defn}\label{def:nonAscending}
    A deformation space $\mathcal{D}$ is \emph{non-ascending} if it is irreducible and there is no $G$-tree $(T,\Omega) \in \mathcal{D}$ containing an edge $e$ with the following properties:
    \begin{enumerate}
        \item Both endpoints of $e$ are in the same $G$-orbit,
        \item The subgroup $\Stab{\Omega}{e}$ is equal to the stabiliser of one endpoint and properly contained in the other.
    \end{enumerate}
\end{defn}

\noindent The following theorem explains our interest in non-ascending deformation spaces:
\begin{thm}\label{thm:GLSurvivingPath}
Let $\mathcal{D}$ be a non-ascending deformation space, and let $\mathcal{F}$ be the subcomplex of $\mathcal{D}$ spanned by surviving trees. For every $T_1, T_2\in \mathcal F$, there exists an elementary deformation from $T_1$ to $T_2$ in which, at each step, the resulting tree is surviving.
\end{thm}

\begin{proof} Such an elementary deformation is a path from $T_1$ to $T_2$ in $\mathcal F$, so we have to show that $\mathcal F$ is connected. This follows from the fact that $\mathcal D$ is connected by construction, and deformation retracts onto $\mathcal{F}$ by \cite[Theorem 7.6]{GL_def_trees}.
\end{proof}

\noindent We conclude the section with a sufficient condition for a deformation space to be non-ascending:
\begin{lemma}\label{lem:treeImpliesNonAscending}
    Suppose $\mathcal{D}$ is an irreducible deformation space containing a $G$-tree~$T$ such that $T/G$ is a tree. Then $\mathcal{D}$ is non-ascending.
\end{lemma}
\begin{proof}
    Towards a contradiction, suppose that $T'$ is a $G$-tree with an edge $e$ as forbidden by Definition \ref{def:nonAscending}. Then the image of $e$ in $T'/G$ is a loop, and in particular the latter is not a tree. Furthermore, the \emph{Betti number} of a $G$-tree (i.e. the rank of the fundamental group of the quotient) is a constant across all trees in a deformation space \cite[Section 4]{GL_def_trees}; hence no $T\in \mathcal{D}$ is such that $T/G$ is a tree, against the hypothesis.
\end{proof}

We can rephrase the requirement on $T/G$ in the above Lemma with the existence of a suitable fundamental domain for the $G$-action.
\begin{defn}[combinatorial fundamental domain]\label{defn:combfunddom}
    Let $(T,\Omega)$ be a $G$-tree; a \emph{combinatorial fundamental domain} is a subtree $K\subseteq T$ such that:
    \begin{enumerate}
        \item for every vertex $v\in\ver T$, $\abs{G\cdot v\cap \ver K}=1$;
        \item for every edge $e\in \edge T$, $\abs{G\cdot e\cap K}=1$.
    \end{enumerate}
\end{defn}

\begin{lemma}\label{lem:cfdTreeQuotient}
    A $G$-tree $(T, \Omega)$ admits a combinatorial fundamental domain if and only if $T/G$ is a tree.
\end{lemma}
\begin{proof}
    If $T/G$ is a tree then any lift of the quotient to $T$ is a combinatorial fundamental domain. Conversely, if $T$ admits a combinatorial fundamental domain $K$ then $T/G\cong K$ is a tree, since the orbits of $T$ are in bijection with $K$.
\end{proof}

\section{Background on Artin groups}\label{sec:cyclicsplit}
In this section, we recall some well known facts on Artin groups, and results from \cite{JonManSar_JSJSpl} which we will need. We also introduce a deformation space for Artin groups, depending only on the union of the conjugacy classes of the generating set. \par 

\begin{defn}\label{defn:artin}
Given a finite, labelled simplicial graph $\Gamma$, let $A_\Gamma$ be the associated Artin group, as in the Introduction. Given a group $A$ and a finite subset $S$, we say that $(A,S)$ is an \emph{Artin system} if $A \cong A_{\Gamma_S}$ for some graph $\Gamma_S$, and the isomorphism sends $S$ bijectively to the vertices of $\Gamma_S$. We say that $S$ is an \emph{Artin generating set} of $A$, and that $\Gamma_S$ is the associated \emph{defining graph} or \emph{Coxeter graph}. Two Artin systems $(A,S)$ and $(A',S')$ are isomorphic if there is an isomorphism $A\cong A'$ mapping $S$ to $S'$. 
\end{defn}

\begin{rem}
The defining graph of an Artin system is well defined up to labelled graph isomorphism. As such, we will freely talk about subgraphs spanned by generators in $S$, implicitly identifying them with vertices in $\Gamma_S$.
\end{rem}

\noindent In the sequel, we will freely make use of the following lemma.

\begin{lemma}\label{lem:isomorphismspreserveags}
    Suppose $(A,S)$ is an Artin system and $\psi : A \rightarrow A'$ is an isomorphism. Then $(A', \psi(S))$ is an Artin system, and $\Gamma_{\psi(S)} \cong \Gamma_S$.
\end{lemma}
\begin{proof}
    By definition of an Artin system, there is an isomorphism $\tau: A_{\Gamma_S} \rightarrow A$ sending the generators of $A_{\Gamma_S}$ corresponding to vertices of $\Gamma_S$ to $S$. Hence, $\psi\tau: A_{\Gamma_S} \rightarrow A'$ is an isomorphism sending the vertices of $\Gamma_S$ to $\psi(S)$. Hence $(A', \psi(S))$ is an Artin system, and $\Gamma_{\psi(S)} \cong \Gamma_S$ as required.
\end{proof}

\begin{defn}\label{defn:dihedral}
If an Artin system $(A,S)$ is such that $\Gamma_S$ is an edge with label $m\ge 3$, then $A$ is called a \emph{dihedral Artin group}. We will call dihedral Artin groups odd or even depending on if $m$ is odd or even. If $S=\{a,b\}$, the centre of $A$ is infinite cyclic, generated by 
\[z_{ab}=\begin{cases}
\Delta_{ab}     &\mbox{ if $m$ is even};\\
\Delta_{ab}^2   &\mbox{ if $m$ is odd},
\end{cases}\]
where $\Delta_{ab}=\mathrm{prod}(a,b,m_{ab})$ is the \emph{Garside element} (see e.g. \cite{BrieskornSaito}).
\end{defn}

\begin{rem}\label{rem:oddConjugation}
    Notice that if $A_S$ is an odd dihedral Artin group, then with the above notation $\Delta_{ab}$ conjugates $a$ to $b$ and vice versa.
\end{rem}

\begin{defn}
Given an Artin system $(A,S)$, a \emph{standard parabolic subgroup} (with respect to $S$) is a subgroup generated by some $U \subseteq S$, which we denote by $A_U$. By a theorem of Van der Lek, $A_U \cong A_{\Gamma_U}$, where $\Gamma_U$ is the induced subgraph of $\Gamma_S$ spanned by the vertices in $U$, and $(A_U,U)$ is an Artin system \cite[Theorem~4.13]{VanDerLek}. We say a subgroup of $A$ is a \emph{parabolic subgroup} (still with respect to $S$) if it is conjugate to a standard parabolic subgroup. We will often write an $S$-parabolic subgroup to mean a parabolic subgroup with respect to $S$.
\end{defn}

\begin{lemma}\cite[Corollary 4.2]{Pparabolics}\label{lem:conjugatingStandardGenerators}
    Suppose $(A,S)$ is an Artin system and~$s, t \in S$. Suppose further that $gsg^{-1} = t$ for some $g \in A$. Then there is a sequence $\{s_i\}_{i=0}^k\subseteq S$ such that:
    \begin{enumerate}
        \item $s_0=t$, $s_k=s$, and for every $i\in\{0,\dots,k-1\}$, $\{s_i,s_{i+1}\}$ spans an odd dihedral Artin $S$-parabolic subgroup;
        \item $g\in\centre{A}{t}\Delta_{s_0s_1}\cdots\Delta_{s_{k-1}s_k}$, where $\centre{A}{t}$ denotes the centraliser of $t$.
    \end{enumerate}
\end{lemma}
\begin{proof}
    The first point is \cite[Corollary 4.2]{Pparabolics} as stated. The second point follows from the first: indeed, by Remark \ref{rem:oddConjugation}, $$(\Delta_{s_0s_1}\cdots\Delta_{s_{k-1}s_k})^{-1} t  (\Delta_{s_0s_1}\cdots\Delta_{s_{k-1}s_k}) = s;$$ hence $g(\Delta_{s_0s_1}\cdots\Delta_{s_{k-1}s_k})^{-1} \in \centre{A}{t}$, and the result follows by right-multiplication by  $\Delta_{s_0s_1}\cdots\Delta_{s_{k-1}s_k}$.
\end{proof}

\subsection{\texorpdfstring{$\Z$}{Z}-splittings of Artin groups} In \cite{JonManSar_JSJSpl}, we characterise when an Artin group splits over virtually cyclic subgroups.

\begin{rem}[One-endedness]\label{rem:one_end}
An Artin group $A$ is one-ended if and only if some (equivalently, any) $\Gamma$ such that $A\cong A_\Gamma$ is connected and has at least two vertices (see e.g. \cite[Remark~2.6]{JonManSar_JSJSpl} for further details). 
\end{rem}

\begin{defn}\label{defn:separating}
    We say that subgraph $\Lambda$ of a simplicial graph $\Gamma$ is \emph{separating}, or that it \emph{separates} $\Gamma$, if the subgraph of $\Gamma$ spanned by $\ver{\Gamma}-\ver{\Lambda}$ is disconnected. If $\Lambda$ consists of a single vertex $v$, we say that $v$ is a \emph{separating vertex}. 
\end{defn}
\noindent The presence of a separating vertex (almost) fully characterises when an Artin group splits over $\Z$, in the sense of Definition~\ref{def_splitoverZ}:

\begin{thm}[{\cite[Theorem A]{JonManSar_JSJSpl}}]\label{thm:no_split}
    Let $A$ be a one-ended Artin group. Then $A$ splits over $\Z$ if and only if, for some (equivalently every) Artin system $(A,S)$,
    \begin{itemize}
        \item $|S|=2$, or
        \item $\Gamma_S$ has a separating vertex.
    \end{itemize}
\end{thm}
\noindent Under the hypothesis of the second bullet of the theorem, there exist $U,V\subsetneq S$ such that $\Gamma_U\cap\Gamma_V$ is a vertex $s\in S$ while $U\cup V=S$. Hence $A=A_U*_{\langle s\rangle}A_V$, and we call such a decomposition a \emph{visual splitting} over a separating vertex.

Theorem~\ref{thm:no_split} motivates the following definition.

\begin{defn}\label{def:bigchunk}
    Given a graph $\Gamma$, a \emph{1--chunk} (also called \emph{big chunk} in \cite{JonManSar_JSJSpl}) is a connected induced subgraph of $\Gamma$ without separating vertices, which is maximal (with respect to inclusion) with these properties. If $\Gamma=\Gamma_S$ for some Artin system $(A,S)$, a \emph{1--chunk parabolic subgroup} is a subgroup of $A$ conjugated to some $A_U$, where $U \subseteq S$ spans a 1--chunk in $\Gamma_S$.
\end{defn}

\noindent The next Theorem allows us to speak of the number and isomorphism types of 1--chunks parabolic subgroups of an Artin group, without any reference to Artin generating sets. We only state the subset of the conclusions that we shall need.
\begin{thm}[{\cite[Theorem~5.6]{JonManSar_JSJSpl}}]\label{thm:iso_invariance}
     Let $\Gamma$ and $\Gamma'$ be finite, connected labelled simplicial graphs. Let $\BC(\Gamma)$ be the set of 1--chunks of $\Gamma$, and define $\BC(\Gamma')$ analogously. Let $\phi\colon A_\Gamma\to A_{\Gamma'}$ be an isomorphism which maps standard generators of $\Gamma$ to conjugates of standard generators of $\Gamma'$. Then there exists a bijection $\phi_\#\colon \BC(\Gamma)\to \BC(\Gamma')$ such that, for every $\Lambda\in \BC(\Gamma)$, $A_{\phi_\#(\Lambda)}$ is a conjugate of $\phi(A_{\Lambda})$, and in particular $A_{\Lambda}\cong A_{\phi_\#(\Lambda)}$.
\end{thm}

\noindent We will use this theorem only in the form of the following immediate corollary. For an Artin system $(A,S)$, call $R_S \coloneq  \{s^a \mid s \in S, a \in A\}$ the \emph{set of reflections}.

\begin{cor}\label{cor:same_parabolics}
    Let $A$ be an Artin group and $S, U \subseteq A$ be Artin generating sets such that $R_S = R_U$. Then $H \le A$ is a 1--chunk parabolic subgroup of $(A,S)$ if and only if it is a 1--chunk parabolic subgroup of $(A,U)$. 
\end{cor}
\begin{proof}
    This follows by applying Theorem \ref{thm:iso_invariance} to the isomorphism $A_{\Gamma_S} \cong A \cong A_{\Gamma_U}$.
\end{proof}

\subsection{A deformation space for Artin groups}\label{sec:defspace_for_artin}
We now describe a deformation space for an Artin system $(A,S)$, only depending on the set of reflections $R_S$, whose trees roughly correspond to maximal visual splittings over separating vertices. Towards proving Theorem~\ref{thmintro:refrigidreduction}, we restrict our attention to one-ended Artin groups.

\begin{defn}[{see \cite[Definition 5.2]{JonManSar_JSJSpl}}]\label{def:crushed_tree}
    Let $A$ be a one-ended Artin group. For every Artin system $(A,S)$ let $\mathcal M_S$ be the graph-of-groups decomposition of $A$ defined as follows.
    \begin{itemize}
        \item The underlying graph $Y$ of $\mathcal M_S$ has one black vertex for every 1--chunk of $\Gamma_S$, and one white vertex for every separating vertex of $\Gamma_S$. 
        \item The vertex group associated to a black vertex (henceforth, a \emph{black vertex group}) is the corresponding standard 1--chunk parabolic subgroup, while the vertex group associated to a white vertex (henceforth, a \emph{white vertex group}) is generated by the corresponding separating vertex, seen as an element of~$S$.
        \item There is an edge between a white vertex and a black vertex if the corresponding separating vertex belongs to the corresponding 1--chunk. The edge group is the same as the white vertex group, which embeds in the black vertex group via the subgraph embedding.
        \item By its construction, it is straightforward to check that $Y$ is a finite tree. 
    \end{itemize} 
    See Figure~\ref{fig:ovJ} for an example. Let $M_S$ be the Bass--Serre tree of $\mathcal M_S$, which is well-defined up to equivariant isometry, as two identifications of $S$ with $\Gamma_S$ can only differ by a graph automorphism of $\Gamma_S$, which descends to an equivariant isometry of $M_S$. Hence the deformation space of $M_S$ is well-defined, and we denote it by~$\crush{D}{S}$.
\end{defn}

\begin{figure}[htp]
\includegraphics[width=\textwidth, alt={Example of how to construct the graph-of-groups decomposition, which has a black vertex for every 1--chunk, stabilised by the associated parabolic subgroup, and a white vertex for every separating vertex, which is only fixed by the cyclic subgroup generated by the vertex.}]{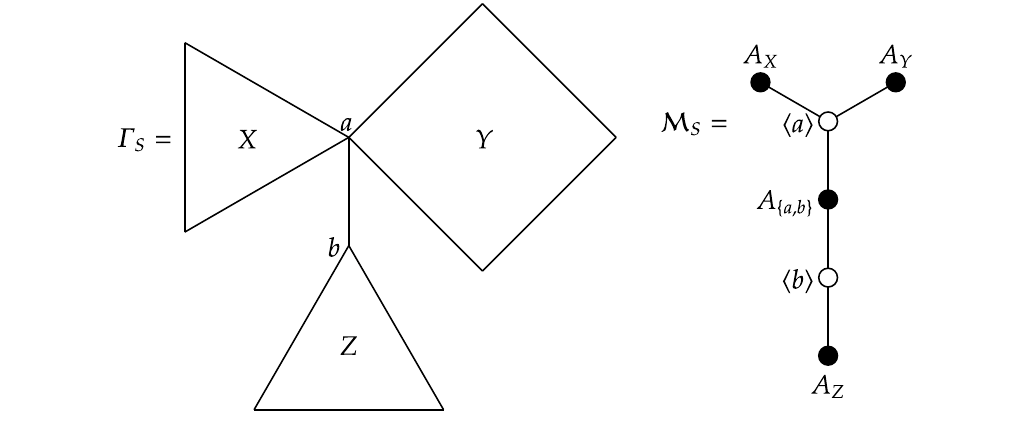}
\caption{Let $(A,S)$ be an Artin system with defining graph $\Gamma_S$ as above (this is the graph from Figure~\ref{fig:strongtwist_new_example}). The separating vertices are $a$ and $b$, and the subsets $X, Y,Z,\{a,b\}\subset S$ each span a 1--chunk in $\Gamma_S$. The decomposition $\mathcal M_S$ has one white vertex for every separating vertex, and one black vertex for every 1--chunk. Notice that every $s\in S$ belongs to some vertex group of $\mathcal M_S$.}
\label{fig:ovJ}
\end{figure}

\begin{rem}[$M_S$ is surviving]\label{rem:JS_survives}
Every orbit of edges in $M_S$ meets the \emph{natural lift} of $\mathcal{M}_S$ (i.e. the subtree of $M_S$ that is isomorphic to $Y$ and whose vertex and edge stabilisers equal the vertex and edge groups in $\mathcal M_S$) exactly once. Hence, under the identification between $Y$ and the natural lift, we can define an elementary collapse directly on the level of the graph-of-groups $\mathcal{M}_S$.

Now, every edge of $\mathcal M_S$ corresponds to a pair $\{v,\Lambda\}$, where $\Lambda$ is a 1--chunk of $\Gamma_S$ and $v\in \Lambda$ separates $\Gamma_S$. If for every such $v$ we collapse one of the adjacent edges of $\mathcal M_S$, we get a graph-of-groups whose Bass-Serre tree is reduced, as two different 1--chunks parabolic subgroups are not contained one in the other. Furthermore, every separating vertex belongs to at least two 1--chunks, so for every edge $e$ of $M_S$ one can find a collapse as above where $e$ survives. See Figure~\ref{fig:JS_survives} for an example.
\end{rem}

\begin{figure}[htp]
    \centering
    \includegraphics[width=\textwidth, alt={Example of how to choose a sequence of collapses to make every given edge survive.}]{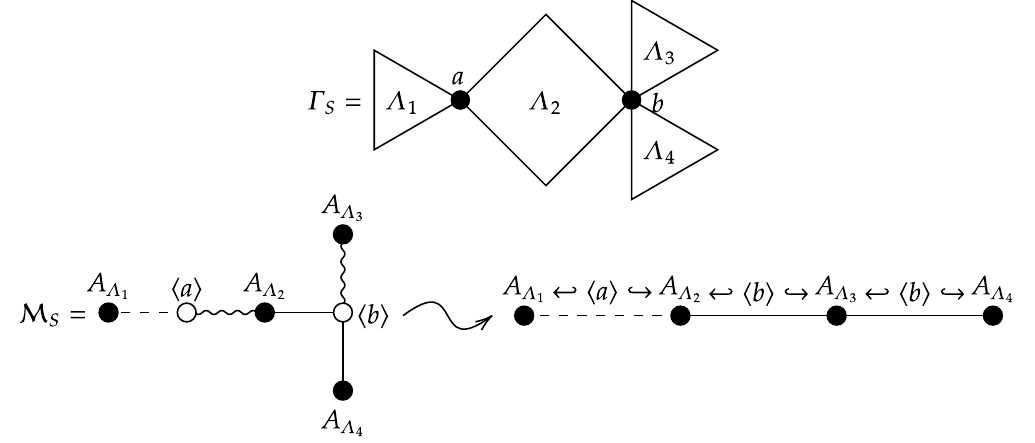}
    \caption{Each $\Lambda_i$ is a 1--chunk of $\Gamma_S$. If we fix an edge of $M_S$ (here, the dashed line represents the corresponding edge in the graph of groups $\mathcal M_S$), we can always collapse one edge of $\mathcal M_S$ for every white vertex (here, the collection of undulated lines) to get a reduced Bass--Serre tree where the orbit of the fixed edge survives.}
    \label{fig:JS_survives}
\end{figure}

\begin{lemma}\label{lem:crushedDependsOnReflections}
    Let $S$ and $U$ be Artin generating sets of $A$ such that $R_S = R_{U}$. Then $\crush{D}{S} = \crush{D}{U}$.
\end{lemma}

\begin{proof}
    We will use the characterisation from Theorem~\ref{thm:forester}, hence it suffices to check that $M_S$ and $M_U$ have the same elliptic subgroups. In turn, by construction elliptic subgroups of $M_S$ are precisely the subgroups of 1--chunk parabolic subgroups of~$S$, and similarly for $U$; hence the lemma follows from Corollary~\ref{cor:same_parabolics}.
\end{proof}

\noindent In view of the Lemma, from now on we shall refer to $\crush{D}{S}$ as $\crush{D}{R}$ where $R=R_S$, in all situations where only the dependence on the reflection set is relevant.

\section{Reducing the twist conjecture}
\label{sec:red twist conj}
This section is devoted to the proof of Theorem~\ref{thmintro:refrigidreduction}. The theorem will follow from an intermediate result involving a generalised version of the strong twist conjecture, which is the main technical tool of this paper (see Theorem~\ref{thm:refrigidreduction}). We first describe some tools to produce new Artin generating sets, namely elementary twists and Dehn Twists.

\subsection{Elementary twists}
For the next definition, we say that an Artin system $(A,S)$ is \emph{indecomposable} if $S$ cannot be partitioned into two non-empty, disjoint subsets $Y,Z$ such that $m_{yz}=2$ for every $y\in Y$ and $z\in Z$. We also say that an Artin system $(A,S)$ is of \emph{spherical type} if the associated \emph{Coxeter group} $A/\nspan{s^2\mid s\in S}$ is finite.
\begin{defn}[{Garside element, \cite{garside}}]
If $(A,S)$ is of spherical type and indecomposable, there is a distinguished element $\Delta_S\in A$, which we call the \emph{Garside element}. For our purposes, it is enough to know that:
\begin{itemize}
\item If $S=\{a\}$, then $\Delta_S=a$.
\item If $S=\{a,b\}$, then $\Delta_S=\Delta_{ab}$ is as in Definition~\ref{defn:dihedral}.
\end{itemize}
\end{defn}
\noindent For the sake of light notation, we will often write $\Delta_{ab}$ instead of $\Delta_{\{a,b\}}$ for the Garside element associated to a dihedral Artin group, as in Definition \ref{defn:dihedral}.

\begin{defn}[Twist equivalence]\label{defn:twist_eq}
    Let $(A,S)$ be an Artin system, and let $J\subseteq S$ be such that $(A_J,J)$ is of spherical type and indecomposable. Let $J^\bot$ be the generators in $S-J$ which commute with $J$. Suppose that $S-(J\cup J^\bot)$ is a disjoint union $B\sqcup C$, where $B$, $C$ are non-empty and any $b\in B$ is not adjacent to any $c\in C$ in $\Gamma_S$ (in other words, $\Gamma_{J\cup J^\bot}$ separates $\Gamma_S$). The \emph{elementary twist} of $B$ around $J$ is the map $\tau \colon S \rightarrow A$ defined by
    $$\tau(s)\coloneq\begin{cases}
        s & \mbox{ if }s\in C\cup J\cup J^\bot;\\
        \Delta_J s \Delta_J^{-1} & \mbox{ if }s\in B.
    \end{cases}$$
    The image $\tau(S)$ is again an Artin generating set for $A$ by \cite[Theorem 4.5]{BMMNtwistconjecture}; moreover, $R_S=R_{\tau(S)}$ by construction. 

    Two Artin generating sets $S, S'$ are \emph{twist equivalent}, and we write $S\sim S'$, if they are connected by a \emph{twist equivalence}, that is, a finite sequence $$S=S_0\xrightarrow[]{\psi_1}S_1\xrightarrow[]{\psi_2}\ldots\xrightarrow[]{\psi_n}S_n=S'$$
    where each $S_i$ is an Artin generating set and each $\psi_i$ is either an elementary twist or the restriction of a conjugation. Two simplicial graphs $\Gamma,\Gamma'$ are twist equivalent, and we write $\Gamma\sim\Gamma'$, if there exist Artin systems $(A,S)$ and $(A,S')$ such that $\Gamma_S=\Gamma$, $\Gamma_{S'}=\Gamma'$, and $S\sim S'$.

    An Artin system $(A,S)$ satisfies the \emph{strong twist conjecture} if every Artin generating set $U \subseteq A$ with $R_S=R_U$ is twist equivalent to $S$. An Artin group $A$ satisfies the strong twist conjecture if every Artin system $(A,S)$ does.
\end{defn}

\noindent For later purposes, we explore in detail a special case of elementary twist.

\begin{ex}[Twist along an edge]\label{ex:odd_twist}
    Let $(A,S)$ be an Artin system, and suppose that $\Gamma_S$ has a separating edge $\{a,b\}$ with label $m_{ab}\ge 3$. Let $S-\{a,b\}=B\cup C$ as in Definition~\ref{defn:twist_eq}, and let $\tau$ be the elementary twist of $B$ along $\{a,b\}$. For every $s\in B\cup\{a,b\}$ let $s'=\Delta_{ab} s \Delta_{ab}^{-1}$, and set $B'=\{s'\}_{s\in B}$, so that $\tau(S)=C\cup \{a,b\}\cup B'$.
    Conjugation by the Garside element $\Delta_{ab}$ permutes $\{a,b\}$, so $s\in B$ and $x\in\{a,b\}$ generate a dihedral subgroup if and only if $s'$ and $x'$ generate a dihedral subgroup with the same label. Hence the defining graph of $(A,\tau(S))$ is as follows:
    \begin{itemize}
        \item The vertex set of $\Gamma_{\tau(S)}$ is $\tau(S)$;
        \item two vertices in $C\cup\{a,b\}$ are joined by an edge with label $m$ in $\Gamma_{\tau(S)}$ if and only if they are joined by an edge with label $m$ in $\Gamma_{S}$;
        \item two vertices $s',t'\in B'\cup\{a,b\}$ are joined by an edge with label $m$ in $\Gamma_{\tau(S)}$ if and only if the corresponding vertices $s,t\in B\cup\{a,b\}$ are joined by an edge with label $m$ in $\Gamma_{S}$.
    \end{itemize}
    Now, if the label of $\{a,b\}$ is even, then $\Delta_{ab}$ commutes with both $a$ and $b$, and therefore $\Gamma_{\tau(S)}\cong \Gamma_{S}$. If otherwise $\{a,b\}$ has odd label, then the conjugation by $\Delta_{ab}$ swaps $a$ and $b$. Informally, the twist along an odd edge $\{a,b\}$ ``slides'' along $\{a,b\}$ every edge of $\Gamma_S$ connecting $\{a,b\}$ and $B$. In particular, if, say $a$ separates $B$ from the rest of $S$ in $\Gamma_S$, then $b$ separates $B'$ from the rest of $\tau(S)$ in $\Gamma_{\tau(S)}$. See Figure~\ref{fig:slide} for some examples.
\end{ex}

\begin{figure}[htp]
    \centering
    \includegraphics[width=\textwidth, alt={An elementary twist along an odd edge.}]{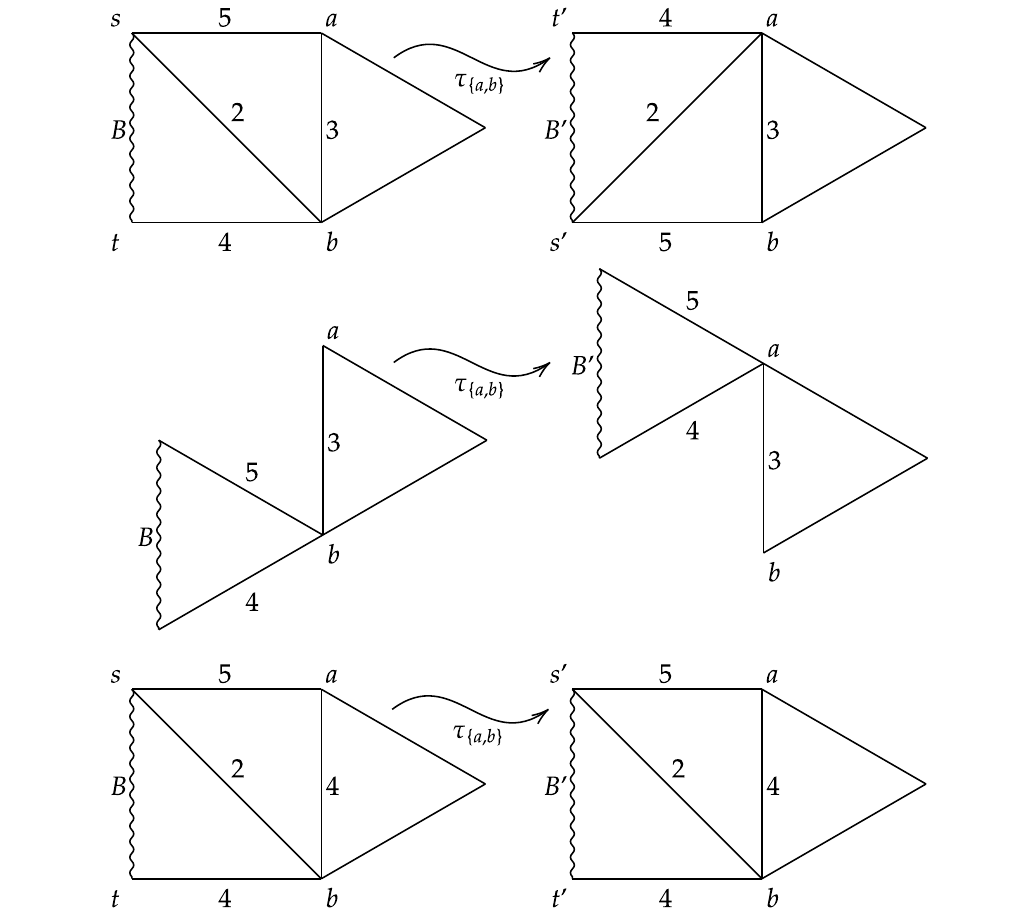}
    \caption{Above, an elementary twist along the odd edge $\{a,b\}$ swaps the roles of $a$ and $b$ in all edges connecting $B$ to $\{a,b\}$. If moreover, say, $b$ is separating (as in the middle row), the twist ``slides'' $B$ along $\{a,b\}$, so that $a$ is now separating in $\Gamma_{\tau(S)}$. Finally, an elementary twist along an even edge preserves the defining graph (as in the bottom row).}
    \label{fig:slide}
\end{figure}

\subsection{Dehn Twists}
We now define another method of producing new Artin generating sets:
\begin{defn}[Dehn Twist]\label{defn:dehn_twist}
    Let $(A,S)$ be an Artin system, and $r \in S$ separate $\Gamma_S$, so $S - \{r\}$ is a disjoint union $B \sqcup C$ where $B, C$ are non-empty and for all $b \in B$ and $c \in C$, $b$ and $c$ are not adjacent in $\Gamma_S$. Let $g\in A_{B\cup \{r\}}\cup A_{C\cup \{r\}}$ centralise $r$. The \emph{Dehn twist} of $B$ about $r$ by $g$ is the map $\delta\colon S \rightarrow A$ defined as follows: $$\delta(s)\coloneq\begin{cases}
        s & \mbox{ if }s\in C\cup\{r\};\\
        g s g^{-1} & \mbox{ if }s\in B.
    \end{cases}$$
\end{defn}

\noindent The following fact follows from general arguments about amalgamated free products, but we prove it for completeness.
\begin{lemma}\label{lem:dehntwistpreservesgraph}
    If $\delta\colon S\to A$ is a Dehn Twist, then $\delta(S)$ is an Artin generating set for $A$, $\Gamma_{\delta(S)}$ is isomorphic to $\Gamma_{S}$, and $R_S=R_{\delta(S)}$.
\end{lemma}

\begin{proof}
    We first prove that $\delta$ can be extended to a homomorphism $\hat\delta\colon A\to A$. Let $i\colon A_{C\cup\{r\}}\to A$ be the inclusion, and let $j\colon A_{B\cup\{r\}}\to A$ map every $x\in A_{B\cup\{r\}}$ to $gxg^{-1}$. Notice that $A_{B\cup \{r\}}\cap A_{C\cup\{r\}}=A_{\{r\}}=\langle r\rangle$ by \cite{VanDerLek}, and $i$ and $j$ coincide on $\langle r\rangle$ since $g$ commutes with $r$. Hence the universal property of amalgamated free products produces a homomorphism $\hat\delta\colon A\to A$ which coincides with $i$ on $A_{C\cup\{r\}}$ and with $j$ on $A_{B\cup\{r\}}$, and therefore extends $\delta$. 

    We now show that $\hat\delta$ is an isomorphism. Suppose first that $g\in A_{B\cup\{r\}}$, so that $\hat\delta$ restricts to an inner automorphism on $A_{B\cup\{r\}}$. Let $\delta'\colon S\to A$ be the Dehn twist of $B$ about $r$ by $g^{-1}$, which as above extends to a homomorphism $\hat\delta'\colon A\to A$, and notice that $\hat\delta'$ restricts to the conjugation by $g^{-1}$ on $A_{B\cup\{r\}}$. Then $\hat\delta'\circ\hat\delta$ is the identity on both $A_{B\cup\{r\}}$ and $A_{C\cup\{r\}}$, so it must be the identity of $A$; similarly $\hat\delta\circ\hat\delta'=\text{id}_A$, so we found an inverse for $\hat\delta$.

    In the case when $g\in A_{C\cup\{r\}}$, the Dehn twist of $B$ about $r$ by $g$ is the composition of the Dehn twist of $C$ about $r$ by $g^{-1}$, which is an isomorphism by the arguments in the previous paragraph, followed by the conjugation of the whole $A$ by $g$; hence $\hat\delta$ is again an isomorphism.

    The above shows that $\delta(S)$ is the image of $S$ under an isomorphism. By Lemma \ref{lem:isomorphismspreserveags} it follows that $\delta(S)$ is an Artin generating set, with $\Gamma_S \cong \Gamma_{\delta(S)}$. Finally, every generator in $S$ is mapped by $\delta$ to an element in its conjugacy class, so $R_S=R_{\delta(S)}$.
    \end{proof}

\begin{defn}\label{def:gentwisteq}
    Two Artin generating sets $S,U$ for $A$ are \emph{generalised twist equivalent}, and we write $S\sim_D U$, if they are connected by a \emph{generalised twist equivalence}, that is, a finite sequence $$S=S_0\xrightarrow[]{\psi_1}S_1\xrightarrow[]{\psi_2}\ldots\xrightarrow[]{\psi_n}S_n=S'$$
    where each $S_i$ is an Artin generating set and each $\psi_i$ is either an elementary twist, the restriction of a Dehn twist, or the restriction of a conjugation. An Artin system $(A,S)$ satisfies the \emph{generalised strong twist conjecture} if for every Artin system $(A,U)$ with $R_S=R_U$ we have that $S\sim_D U$. An Artin group $A$ satisfies the generalised strong twist conjecture if every Artin system $(A,S)$ does.
\end{defn}
\begin{rem}\label{rem:gtc_implies_wtc}
    Notice that if $S\sim_D U$ and $R_S=R_U$, then $\Gamma_S\sim \Gamma_{U}$. To see this it is enough to check that if $S'$ is related to $S$ by an elementary twist, Dehn twist or conjugation, then $\Gamma_{S} \sim \Gamma_{S'}$. The elementary twist and conjugation cases are clear. Furthermore, Dehn twisting does not change the defining graph, by Proposition~\ref{lem:dehntwistpreservesgraph}. It follows that the generalised strong twist conjecture implies the weak twist conjecture from the introduction (see Conjecture~\ref{conjintro_WTC}).
\end{rem}

\subsection{The strong twist conjecture reduces to 1-chunks}\label{sec:proof_of_refrigid}
We are now ready to state the main result of this section, which implies Theorem~\ref{thmintro:refrigidreduction}:

\begin{thm}\label{thm:refrigidreduction}
    Let $(A,S)$ be an Artin system, with $A$ one-ended. Suppose that, whenever $X\subseteq S$ spans a 1--chunk, the Artin system $(A_X,X)$ satisfies the strong twist conjecture. Then $(A,S)$ satisfies the generalised strong twist conjecture, and in particular the weak twist conjecture by Remark~\ref{rem:gtc_implies_wtc}.
\end{thm}

\noindent For the rest of the section, unless otherwise stated, we work under the following assumption:
\begin{notation}\label{notation:ASR} Let $A$ be an Artin group, with $A$ one-ended. Let $R$ be a fixed reflection set associated to some Artin generating set, and let $\crush{D}{R}$ be the corresponding deformation space. \end{notation}

\noindent We start by observing the following:
\begin{lemma}
    Every $A$-tree in $\crush D {R}$ admits a combinatorial fundamental domain, as in Definition~\ref{defn:combfunddom}.
\end{lemma}

\begin{proof} Take $S$ an Artin generating set such that $R_S = R$. By construction $M_S/A$ is a tree, and therefore so is $T/A$ for every $T\in \crush D {R}$ by invariance of the Betti number across the deformation space. Then the statement follows from Lemma~\ref{lem:cfdTreeQuotient}. \end{proof}

\begin{cor}\label{lem:crushNonAscending}
    If $S$ is an Artin generating set with $R_S = R$ and $\Gamma_S$ has at least two 1--chunks, then $\crush D {R}$ is non-ascending.
\end{cor}
\begin{proof}
    By Lemma~\ref{lem:treeImpliesNonAscending}, it is enough to show that, if $\Gamma_S$ has at least two 1--chunks, then $\crush D {R}$ is irreducible, and in turn it suffices to exhibit two loxodromic elements for the action on $M_S$ whose commutator is again loxodromic. Let $X,Y\subseteq S$ span different 1--chunks, let $a\in X-Y$ and $b\in Y-X$, and set $g=ab$ and $h=a^2b$, which are both loxodromic as they they are products of two elliptic elements with disjoint stable fixed point sets (this is a standard ping-pong type argument, see e.g. \cite{pingpong}). Their commutator $abab^{-1}a^{-2}$ is conjugate to $[b,a]=b(ab^{-1}a^{-1})$, which again is a product of elliptic elements with disjoint stable fixed point sets.
\end{proof}

\begin{lemma}
    \label{lem:edges stab reduced trees}
    Let $\mathcal F\le\crush D {R}$ be the subcomplex spanned by surviving trees, and $S$ be an arbitrary Artin generating set such that $R_S = R$. For every $(T,\Omega)\in\mathcal F$,
    \begin{enumerate}
        \item for every $e\in\edge T$, $\Stab{\Omega}{e}$ is a cyclic $S$-parabolic subgroup;
        \item for every $v \in \ver T$, $\Stab{\Omega}{v}$ is an $S$-parabolic subgroup.
    \end{enumerate}
\end{lemma}

\begin{proof} 
    For the sake of light notation, we will identify $A$-trees with the underlying tree. The claim holds for $M_S$ by how it was constructed in Definition~\ref{def:crushed_tree}. If $\Gamma_S$ only has one 1--chunk then $\crush{D}{R}$ consists of a single point, and there is nothing to prove. Otherwise we can assume that $\crush{D}{R}$ is non-ascending, by Corollary~\ref{lem:crushNonAscending}.
    
    Let ${M}'$ be a reduced tree obtained  from $M_S$ by a finite sequence of elementary collapses, which exists as observed in Remark~\ref{rem:reduced JSJ tree}. By Lemma \ref{lem:stabsOfCollapse}, each collapse preserves the set of stabilisers of those edges that are not collapsed, and every vertex stabiliser in the collapsed tree appeared as a vertex stabiliser in the original tree. Therefore the claim holds for ${M}'$. If $T$ is a reduced tree, then $T$ and ${M}'$ are related by a finite sequence of slide moves~\cite[Theorem 7.2]{GL_def_trees}, which preserve the set of edge and vertex stabilisers (notice that the theorem requires the deformation space to be non-ascending). In particular, the claims holds for each reduced tree.\medskip
    
    Finally, let us deal with the general case in which $T$ is a surviving tree. By definition of surviving, for every edge~$e$ of $T$, there exists a reduced tree $T'$ and a finite sequence of elementary collapses $T\to T'$ that do not collapse~$e$. Again, because elementary collapses preserve the set of stabilisers of edges that are not collapsed, the claim for edges holds for~$T$. \par

    Let us now turn to the claim for vertices. We fix a reduced tree~$(T_0,\Omega_0)$ and a finite sequence of elementary collapses $\phi\colon T\to T_0$. We decompose $\phi$ as $\phi_1\circ\cdots\circ\phi_k$, where we set $T=T_k$ and, for every $i$, $\phi_i\colon T_i\to T_{i-1}$ is an elementary collapse. By induction on $k\in\mathbb N$, we shall prove that, for every vertex $x$ of $T_i$, $\Stab \Omega x$ is an $S$-parabolic subgroup. For the base case, $i=0$ so the result holds as $T_0$ is reduced. Now suppose the claim holds for $T_{i-1}$ and consider $\phi_i\colon T_i \to T_{i-1}$ a single elementary collapse: there is an edge $e=\{u,v\}$ of $T_i$ with $\Stab{\Omega_i} {u}\le\Stab{\Omega_i}{v}$ such that $\phi_i$ is the elementary collapse of the edge~$e$. Let $x$ be a vertex of $T_i$. If $x$ does not belong to the orbit of $u$ or $v$, then $\phi_i$ is injective on $\phi_i^{-1}(\phi_i(x))$ and $\Stab{\Omega_i}{x}=\Stab{\Omega_{i-1}}{\phi(x)}$ by Lemma~\ref{lem:stabsOfCollapse}. Because the claim holds for $T_{i-1}$, $\Stab{\Omega_{i-1}}{\phi(x)}$ is an $S$-parabolic subgroup. Otherwise, up to replacing $x$ with an $\Omega_i$-translate of $x$, we may assume that $x=v$ or $x=u$. If $x=v$, then $\Stab{\Omega_i} {x}=\Stab{\Omega_{i-1}}{\phi_i(e)}$ again by Lemma~\ref{lem:stabsOfCollapse}. By the inductive hypothesis, $\Stab{\Omega_{i-1}}{\phi_i(e)}$ is an $S$-parabolic subgroup. If $x=u$, then $\Stab{\Omega_i} {x}=\Stab{\Omega_i} {e}$ (since $\Stab{\Omega_i} {u}\le\Stab{\Omega_i}{v}$ and the action is without inversion). By the claim on the edges, it follows that $\Stab{\Omega_i}{e}$ is an $S$-parabolic subgroup. This concludes the proof of the claim for vertices.
\end{proof}

\begin{defn}[$S$-tree]
    \label{def:Gamma-tree}
    Let $(A,S)$ be an Artin system; an $A$-tree $(T,\Omega)$ is an \emph{$S$-tree} if there exists a combinatorial fundamental domain $K\subseteq T$ for the action $\Omega$ such that:
    \begin{enumerate}
        \item for every $x\in\ver K\cup\edge K$, $\Stab{\Omega}{x}$ is a standard $S$-parabolic subgroup;
        \item every element of $S$ fixes a vertex of $K$.
    \end{enumerate} 
\end{defn}

\noindent The following two propositions, which are the core arguments underlying Theorem~\ref{thm:refrigidreduction}, show that performing an elementary deformation on an $S$-tree produces an $S'$-tree, where $S'\sim_D S$.
In the proofs we will follow the strategy of~\cite[Lemma 4.8 and Lemma 4.9]{Jones_VF}.

\begin{prop}
    \label{prop:collapse Gamma-tree}
    Let $S$ be an Artin generating set with $R_S = R$ and $(T,\Omega) \in \crush{D}{R}$ be an $S$-tree, and let $(T',\Omega')\in\crush{D}{R}$ be an $A$-tree that is obtained from $(T,\Omega)$ by collapsing one edge. Then  $(T',\Omega')$ is an $S$-tree.
\end{prop}

\begin{proof}
    Let $\phi\colon T \to T'$ be the collapse map, and let $K\subseteq T$ be a combinatorial fundamental domain for $\Omega$ that makes $(T,\Omega)$ into an $S$-tree. Since $\phi$ is $A$-equivariant and $K$ contains exactly one edge for every orbit, there exists a unique edge $e\in K$ that is collapsed, say with endpoints $u$ and $v$ such that $\Stab{\Omega}{u}\le\Stab{\Omega}{v}$. 
    
    We claim that $K'=\phi(K)$ makes $(T',\Omega')$ into an $S$-tree, that is, that $K'$ satisfies the conditions of Definition~\ref{def:Gamma-tree}. Firstly, $K'$ is a combinatorial fundamental domain, by equivariance of $\phi$. 
    
    Let us now show that stabilisers of simplices in $K'$ are standard $S$-parabolic subgroups. Because $\phi$ restricts to a bijection between $K- e$ and $ K'-\{\phi(e)\}$, we only need to check that $\Stab{\Omega'}{\phi(e)}$ is a standard $S$-parabolic subgroup. In this case Equation~\eqref{eq:stab_phi(e)} gives $\Stab{\Omega'}{\phi(e)}=\Stab{\Omega}{v}$, which by hypothesis is a standard $S$-parabolic subgroup. 

    Finally, since $K$ makes $(T, \Omega)$ an $S$-tree, every $s\in S$ fixes some vertex $x\in\ver K$. Since $\phi$ is $A$-equivariant,  $s$ must also fix $\phi(x) \in \ver{\phi(K)} = V(K')$.
\end{proof}

\begin{prop}
    \label{prop:expansion Gamma-tree}
    Let $(T,\Omega),(T',\Omega')\in\mathcal F$ be surviving trees such that $(T,\Omega)$ is obtained from $(T',\Omega')$ by an elementary collapse. If $S$ is an Artin generating set with $R_S = R$ and  $(T,\Omega)$ is an $S$-tree, then there is an Artin generating set $S'\subseteq A$ such that $S\sim_D S'$ and $(T',\Omega')$ is an $S'$-tree.
\end{prop}

\begin{proof}
     We denote by $\phi\colon T'\to T$ the collapsing map. Let $K\subseteq T$ be a combinatorial fundamental domain for $\Omega$ that makes $(T,\Omega)$ into an $S$-tree. Throughout, we use without comment that, by Lemma~\ref{lem:edges stab reduced trees}, every edge and vertex stabiliser in $T$ and $T'$ is an $S$-parabolic subgroup.
     
     Since $K$ is a combinatorial fundamental domain and $\phi$ is $A$-equivariant, there exists an edge $e=\{u,v\}$ of $T'$ that gets collapsed to a vertex $\phi(e)\in K$, say with $\Stab{\Omega'}{u}\le\Stab{\Omega'}{v}$. Now consider the chain
    \begin{equation}
    \label{eqn:chain of stabs}
    \Stab{\Omega'}{e}=\Stab{\Omega'}{u}\le\Stab{\Omega'}{v}=\Stab{\Omega}{\phi(e)}\le A,
    \end{equation}
    where the equality is Equation~\eqref{eq:stab_phi(e)}. This is a chain of inclusions of $S$-parabolic subgroups of $A$, so by~\cite[Theorem 1.1]{blufstein2023parabolic} $\Stab{\Omega'}{e}$ is a parabolic subgroup of $\Stab{\Omega'}{v}$  (meaning that $\Stab{\Omega'}{e}$ is conjugated \emph{by an element of $\Stab{\Omega'}{v}$} to a standard $S$-parabolic subgroup \emph{contained inside} $\Stab{\Omega'}{v}$).  Up to replacing $e$ by a $\Stab{\Omega'}{v}$-translate of it, we may therefore assume that $\Stab{\Omega'}{e}$ is a standard $S$-parabolic subgroup of $\Stab{\Omega'}{v}$, hence of $A$.

    Let us now decompose the fundamental domain $K$ as $\bigcup_{i=1}^nK_i$, where the various $K_i$ are the maximal subtrees of $K$ having $\phi(e)$ as a vertex of valence one, so that $K_i\cap K_j=\{\phi(e)\}$ whenever $i\neq j$. Such decomposition induces subsets $\{S_i\}_{i=0}^n$ of~$S$, where $A_{S_0}=\Stab{\Omega}{\phi(e)} = \Stab{\Omega'}{v}$ and, for every $i\in\{1,\dots,n\}$ and every $s\in S_i$, the subtree of $K$ fixed by $s$ is contained in $K_i-\{\phi(e)\}$. 
    
    \begin{claim}\label{claim:partition}
        The union $\bigcup_{i=0}^n S_i$ is a partition of $S$.
    \end{claim} 

    \begin{claimproof}[Proof of Claim~\ref{claim:partition}]
    Let $s\in S$. By definition of $S$-tree, $s$ fixes a vertex of $K$. If $s$ fixes $\phi(e)$, then $s\in S_0$; if not, then there exists $i\in\{1,\dots,n\}$ such that $s$ fixes a vertex $w_i$ of $K_i$, so $s \in S_i$. If there existed a different index $j$ such that $s$ fixes a vertex $w_j$ of $K_j$, then $s$ would fix the unique path between $w_i$ and $w_j$, hence $\phi(e)$, giving a contradiction. This shows that the $S_i$'s are all disjoint.\end{claimproof}

    \noindent The tree structure of $K$ gives a decomposition of $A$ as an amalgamated product of the factors $\{A_{S_0\cup S_i}\}_{i=1}^n$ over the subgroup $A_{S_0}$. 
    
    \begin{claim}\label{claim:dec}
        The decomposition $A_{S_1\cup S_0}*_{A_{S_0}}\cdots *_{A_{S_0}}A_{S_n\cup S_0}$ is a visual splitting for $(A,S)$.
    \end{claim}

    \begin{claimproof}[Proof of Claim~\ref{claim:dec}]
    Because $\bigcup_{i=0}^nS_i=S$, it is sufficient to show that, for every distinct indices $i,j\in\{1,\dots,n\}$, for every $s\in S_i$ and for every $t\in S_j$, $A_{\{s,t\}}\cong F_2$. To this end, let us observe that $\Stab{\Omega}{\phi(e)}$ is a standard $S$-parabolic subgroup, so $\Stab{\Omega}{\phi(e)}\cap \langle s\rangle=\Stab{\Omega}{\phi(e)}\cap \langle t\rangle=\{1\}$ by \cite{VanDerLek}. Hence $s$ and $t$ have disjoint stable fixed-point sets, in particular separated by $\phi(e)$. Notice that $\phi(e)$ is not fixed by any non-trivial power of $s$ or $t$. Hence, a standard ping-pong argument now shows that $s$ and $t$ generate a non-abelian free group, as required.\end{claimproof}

    We now define simultaneously an Artin generating set $S'\subseteq A$ that is generalised twist-equivalent to $S$, and a subtree $K'\subseteq T'$ that makes $(T',\Omega')$ into an $S'$-tree. To be precise, we set $S_0'=S_0$, and we shall inductively replace each $S_i$ with some $S_i'$ (which we will shortly define) for every $i=1,\ldots,n$. Letting $U_{i}=\bigcup_{j\le i}S_j'\cup \bigcup_{j> i}S_j$ for all $i$ from $0$ to $n$, we will define the $S_i'$ to have the following property: two generators $a,b\in U_i$ may generate a dihedral Artin group only if either they both belong to some $S_j$ (or $S_j'$), or if one of them is in $S_0$. This property ensures that every $U_i$ will give an amalgamated product decomposition over $A_{S_0'}$ for $A$. Furthermore, at each step we shall show that $U_{i-1}\sim_D U_{i}$. Then the required $S'$ will be $U_n$, which will therefore be generalised twist equivalent to~$S$, which is $U_0$.

    Let us define $L_i'\coloneqq\overline{\phi^{-1}(K_i-\{\phi(e)\})}\subseteq T'$, where $\overline{X}$ represents the closure of $X$. Because $\phi$ is injective outside of the orbit of $e$, $L'_i$ is isomorphic to $K_i$ as a graph and contains some $w_i\in \phi^{-1}(\phi(e))$ as a vertex of valence one. The next Claim shows that  $w_i$ is either $v$ or a $\Stab{\Omega'}{v}$-translate of $u$:

    \begin{claim}\label{claim:preimage_phi_e}
        The only vertices in $\phi^{-1}(\phi(e))$ are $v$ and all $\Stab{\Omega'}{v}$-translates of~$u$.
    \end{claim}
    \begin{claimproof}
         To see this, first notice that the only edges in $\phi^{-1}(\phi(e))$ are in the $\Omega'$-orbit of $e$, since $\phi(e)$ has no edges and $\phi$ is injective on edges not in the orbit of $e$. It follows that the vertices of $\phi^{-1}(\phi(e))$ are in the $\Omega'$-orbit of $u$ and $v$, since every vertex of $\phi^{-1}(\phi(e))$ must be the endpoint of one of its edges.
         
         Consider an arbitrary vertex $hx$ of $\phi^{-1}(\phi(e))$, where $h \in A$ and $x \in \{u, v\}$. Notice that $\phi(e) = \phi(hx) = h\phi(x) = h\phi(e)$, where the second equality is by equivariance of $\phi$. So $h \in \Stab{\Omega}{\phi(e)}$ and in fact the vertices of $\phi^{-1}(\phi(e))$ must be in the $\Stab{\Omega}{\phi(e)}$-orbit of $u$ and $v$. Finally, $\Stab{\Omega'}{v}=\Stab{\Omega}{\phi(e)}$ by Equation \eqref{eq:stab_phi(e)}, so the vertices of $\phi^{-1}(\phi(e))$ are in the $\Stab{\Omega'}{v}$-orbit of $u$, or the $\Stab{\Omega'}{v}$-orbit of $v$, which is just $v$.
    \end{claimproof}
    
    We now wish to define $K_i' \subset T'$ to be subtrees which we will eventually use to build the fundamental domain $K'$, analogously to how $K = \bigcup_{i=1}^n K_i$. We define $S_i'$ and $K_i'$ together, to ensure that eventually $K'$ will be a fundamental domain making $T'$ and $S'$-tree. If $w_i$ is either $v$ or $u$, we simply set $K_i'\coloneqq L_i'\cup e$ and $S_i'\coloneqq S_i$. This way $U_{i-1}=U_{i}$, and there is nothing to prove. Let us now assume that $w_i=h\cdot u$ for some $h\in\Stab{\Omega'}{v}-\Stab{\Omega'}{u}$. In this case we must take some care in defining $K_i'$, since we want $K'$ to be connected, and control the point stabilisers.  Let $f$ be the edge of $L_i'$ that has $h\cdot u$ as an endpoint. Because $\phi\colon T'\to T$ restricts to an isometry $L_i'\to K_i$, $f$ is not collapsed, so we have that $\Stab{\Omega'}{f}=\Stab{\Omega}{\phi(f)}$. In particular, the latter is a cyclic standard $S$-parabolic subgroup, say generated by some $s\in S$. In fact, $$s \in \Stab{\Omega'}{f} \leq \Stab{\Omega'}{h \cdot u} \leq \Stab{\Omega'}{v} = A_{S_0}.$$ In turn, the edge $h^{-1}\cdot f$ has $u$ as an endpoint, so $h^{-1}\Stab{\Omega'}{f}h$ is an $S$-parabolic subgroup of $A$ contained in $\Stab{\Omega'}{u}$, hence an $S$-parabolic subgroup of $\Stab{\Omega'}{u}$ by~\cite[Theorem 1.1]{blufstein2023parabolic}. Therefore, up to replacing $h$ by an element of $h\Stab {\Omega'}{u}$ (which does not change the fact that $w_i=hu$, or that $h \in \Stab{\Omega'}{v} = A_{S_0}$), we may assume that $h^{-1}\Stab{\Omega'}{f}h$ is a standard $S$-parabolic subgroup, say generated by $t\in S$. In order that $K'$ is connected, we set $K_i'\coloneqq h^{-1}\cdot L_i'\cup e$. See Figure~\ref{fig:lift_of_CFD} for an intuition.

    \begin{figure}[htp]
        \centering
        \includegraphics[width=0.75\linewidth]{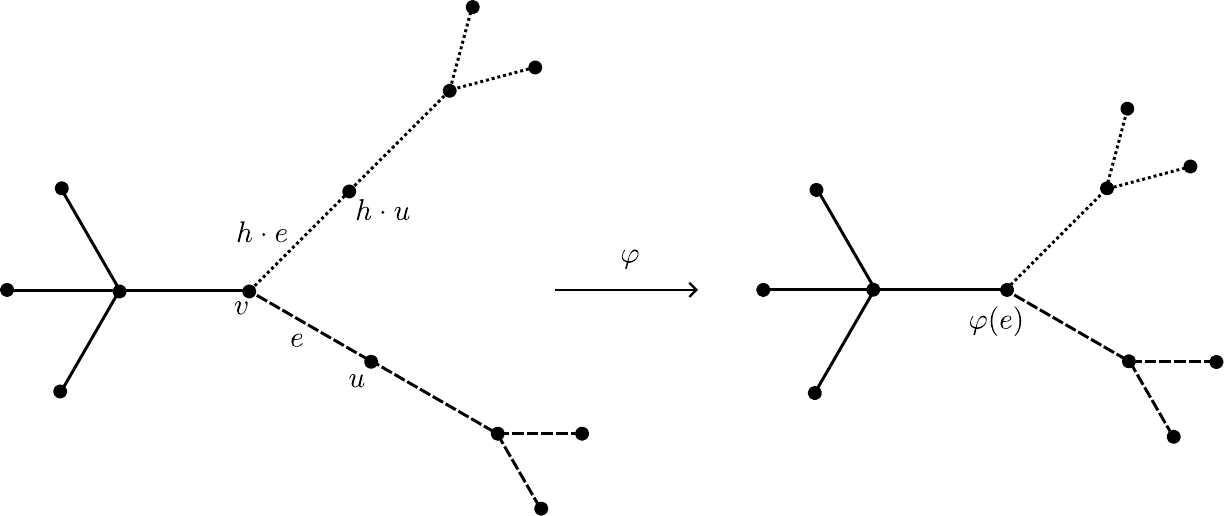}
        \caption{Dotted: $L'_i\cup h\cdot e$ (on the left) and $K_i$ (on the right). Dashed: $K_i'=h^{-1}\cdot L'_i\cup e$ (on the left) and $h^{-1}\cdot K_i$ (on the right).}
        \label{fig:lift_of_CFD}
    \end{figure}

    Since $h^{-1}$ conjugates $s$ to $t$ in the standard parabolic subgroup $A_{S_0}$, it follows by direct application of Lemma~\ref{lem:conjugatingStandardGenerators} that there is a sequence $\{s_i\}_{i=0}^k\subseteq S_0$ such that:
    \begin{enumerate}
        \item\label{item:si} $s_0=t$, $s_k=s$, and for every $i\in\{0,\dots,k-1\}$, $\{s_i,s_{i+1}\}$ spans an odd dihedral Artin $S$-parabolic subgroup;
        \item \label{item:structure of h} $h^{-1}\in\centre{A_{S_0}}{t}\Delta_{s_0s_1}\cdots\Delta_{s_{k-1}s_k}$, where $\centre{A_{S_0}}{t}$ denotes the centraliser of $t$ in $A_{S_0}$.
    \end{enumerate}

        Before defining $S_i'$, we observe the following fact:

    \begin{claim}\label{claim:seps}
        For every $x\in S_i$ and for every $r\in S_0$, if $A_{\{x,r\}}$ is of spherical type, then $r=s$.
    \end{claim}

\begin{proof}[Proof of Claim~\ref{claim:seps}]
Let $x\in S_i$ and let $r\in S_0-\{s\}$. Because $x\in S_i$, the stable fixed point set of $x$ intersects $L_i'$ and does not contain $f$. Indeed, if the stable fixed point set of $x$ contained $f$, then a non-trivial power of $x$ would belong to $\Stab{\Omega'}{f}=\langle s\rangle \le A_{S_0}$, against the fact that $A_{S_i}\cap A_{S_0}=\{1\}$~\cite{VanDerLek}.  On the other hand, $r$ fixes $v$ but not $f$, so its stable fixed-point in $T'$ lies on the opposite side of $f$ with respect to $L_i'$. Again, as in the proof of Claim~\ref{claim:dec}, a standard ping-pong argument shows that $A_{\{x,r\}}\cong F_2$.\end{proof}

    Set $S_i'\coloneqq h^{-1}S_ih$, and recall that we defined $U_{i-1}=\bigcup_{j<i}S_j'\cup \bigcup_{j\ge i}S_j$ and $U_{i}=\bigcup_{j\le i}S_j'\cup \bigcup_{j> i}S_j$. Consider the map $\alpha_i\colon U_{i-1}\to U_{i}$ which maps $a\in U_{i-1}$ to
    \[\alpha_i(a)=\begin{cases}
        h^{-1}ah&\mbox{if }a\in S_i;\\
        a&\mbox{ otherwise}.
    \end{cases}\]

    To conclude that $U_{i-1}$ and $U_{i}$ are generalised twist equivalent, it is enough to prove the following:

    \begin{claim}\label{claim:twist_by_steps}
        $\alpha_i$ is a sequence of elementary twists and a Dehn twist.
    \end{claim}
    
    \begin{claimproof}[Proof of Claim~\ref{claim:twist_by_steps}]
    Recall by Item~\eqref{item:structure of h} that $h^{-1}\in\centre{A_{S_0}}{t}\Delta_{s_0s_1}\cdots\Delta_{s_{k-1}s_k}$. For $l \in \{0 \dots k\}$ we now define 
    $$V_l = \bigcup_{j<i}S_j'\cup (\Delta_{s_{k-l}s_{k-l+1}} \dots \Delta_{s_{k-1}s_k}) S_i (\Delta_{s_{k-l}s_{k-l+1}} \dots \Delta_{s_{k-1}s_k})^{-1} \cup \bigcup_{j > i}S_j,$$
    so in particular $V_0 = U_{i-1}$. For the sake of light notation, we define 
    $$W_l = (\Delta_{s_{k-l}s_{k-l+1}} \dots \Delta_{s_{k-1}s_k}) S_i (\Delta_{s_{k-l}s_{k-l+1}} \dots \Delta_{s_{k-1}s_k})^{-1}.$$
    We will show by induction that, for $l \in \{0 \dots k-1\}$, $V_{l+1}$ differs from $V_l$ by an elementary twist, and $s_{k-l}$ separates $W_l$ from the other standard generators in $V_l$. The base case follows from the fact that $s_k = s$ is separating by Claim \ref{claim:seps}. For the inductive step, we notice that $s_{k-l}$ separates $W_l$ from the other standard generators of $V_l$, so,  by Item \eqref{item:si}, $s_{k-l}$ and $s_{k-l-1}$ span a separating odd edge in $\Gamma_{V_l}$. It follows that $V_{l+1}$ differs from $V_l$ by a twist along this edge. This twist has the effect of ``sliding'' the $\Gamma_{W_l}$ subgraph along the edge $\{s_{k-l}, s_{k-l-1}\}$ in $\Gamma_{V_l}$ (see Example \ref{ex:odd_twist}), so $s_{k-l-1}$ separates $W_{l+1}$ from the other standard generators in $V_{l+1}$, as required.

    Finally, by the structure of $h$ described by Item \eqref{item:structure of h}, $U_i$ is obtained from $V_k$ by conjugating the generators of $W_k$ by an element of $\centre {A_{S_0}}{t}$. By the previous inductive argument, $t = s_0$ separates $W_k$ from the other standard generators in $V_k$, so this partial conjugation is a Dehn twist. We have shown that $V_k$ is obtained from $V_0 = U_{i-1}$ by a series of edge twists. Furthermore, $U_i$ is obtained from $V_k$ by a Dehn twist, and in particular $U_i$ is an Artin generating set by Lemma~\ref{lem:dehntwistpreservesgraph}. This proves the claim. 
    \end{claimproof}

\noindent Let $K'=\bigcup_i K_i'\subseteq T'$. We are now left to show that $K'$ makes $(T',\Omega')$ into an $S'$-tree. $K'$ has exactly one more edge than $K$ (that is, the edge $e$), thus it is a finite subgraph of~$T'$, and it is connected by construction of the $K'_i$'s, each of which is connected and contains $e$. It follows that $K'$ is a finite subtree of~$T'$. For the sake of clarity, we break down the remaining part of the proof into smaller claims.

    \begin{claim}\label{claim:combdom}
        $K'$ is a combinatorial fundamental domain for $(T',\Omega')$.
    \end{claim}

    \begin{claimproof}[Proof of Claim~\ref{claim:combdom}]
Let $x\in\ver {T'}$. We want to show that the $\Omega'$-orbit of $x$ intersects $K'$ only once. If $\phi(x)$ lies in the $\Omega$-orbit of $\phi(e)$, then $x$ lies in the $\Omega'$-orbit of either $u$ or $v$, which intersect each $K_i'$ only at $u$ and $v$ by how we constructed $K_i'$. Hence $\lvert A\cdot x\cap\ver {K'}\rvert=1$ since $u$ and $v$ belong to distinct $\Omega'$-orbits by Definition~\ref{def:elementarycollapse}. If $\phi(x)$ is not in the orbit of $\phi(e)$, then there exist $k\in K-\{\phi(e)\}$ and $g\in A$ such that $\phi(x)=g\cdot k$. Because $\phi$ is injective away from the orbit of $e$, $\phi^{-1}(k)=k'$ for some $k'$ that either belongs to $K'$ or such that a $\Stab{\Omega'}{v}$-translate of it belongs to $K'$. In both cases, the orbit of $k'$ intersects $K'$. Because $\phi$ is $A$-equivariant, $\phi(x)=g\cdot\phi(k')$ implies that $\phi(x)=\phi(g\cdot k')$; because $\phi$ is injective on these points, it follows that $x=g\cdot k'$. It follows that $K'$ intersects the orbit of $x$ precisely once.  

    Now observe that $\edge{K'}$ contains at least one edge in each $\Omega'$-orbit, since every edge in $T'$ is either in the orbit of $e$ or is sent to an edge of $T$ by the collapsing map, and either way there is a corresponding edge in $K'$ by construction. If $\edge{K'}$ contained two distinct edges in the same orbit, then by looking at the endpoints we could find two distinct vertices of $K'$ in the same orbit, contradicting our conclusion for vertices.
\end{claimproof}

    \begin{claim}\label{claim:stab_S'}
        For every $x\in\ver {K'}\cup \edge {K'}$, $\Stab{\Omega'}{x}$ is a standard $S'$-parabolic subgroup.
    \end{claim}

\begin{claimproof}[Proof of Claim~\ref{claim:stab_S'}]
    We have that $\Stab{\Omega'}{v}=\Stab{\Omega}{\phi(e)}=A_{S_0}$ by construction of $K'$; moreover, as argued in Equation~\eqref{eqn:chain of stabs},  $\Stab{\Omega'} u=\Stab{\Omega} e$ is a standard $S$-parabolic subgroup inside $A_{S_0}$. Since $S_0=S_0'$, both stabilisers are standard $S'$-parabolic subgroups as well.\par
    Let now $x\in\ver{K'}-\{u,v\}$. From injectivity of $\phi$ away from the orbit of $e$ and $A$-equivariance, it follows that $\Stab{\Omega'}{x}=\Stab{\Omega}{\phi(x)}$. By construction of $K'$, there exists $h\in\Stab{\Omega'}{v}$ and $i\in\{1,\dots,n\}$ such that $\phi(x) \in h^{-1}\cdot K_i$. Because $K$ makes $(T,\Omega)$ into an $S$-tree, there exists $U\subseteq S_i$ such that $\Stab{\Omega}{h\cdot \phi(x)}=A_U$, that is, $\Stab{\Omega}{\phi(x)}=h^{-1}A_U h$, which is a standard $S'$-parabolic subgroup of $A$, since $h^{-1}Uh\subseteq h^{-1}S_ih=S_i'$.\par
    For each $x \in \edge{K'}$, $\Stab{\Omega'}{x}$ is the intersection of the stabilisers of the endpoints. The endpoints are in $\ver{K'}$, so the conclusion follows from the fact that, for any Artin system, the intersection of two standard parabolic subgroups is a standard parabolic subgroup \cite{VanDerLek}.
\end{claimproof}

    \begin{claim}\label{claim:S'fix}
        Every element $s'\in S'$ fixes a vertex of $K'$.
    \end{claim}

\begin{claimproof}[Proof of Claim~\ref{claim:S'fix}]
    Every $s'\in S_0'$ fixes $e$, so it fixes its endpoints. If instead $s'\in S'-S_0'$, there exist $h\in\Stab{\Omega'}{v}$, $i\in\{1,\dots,n\}$ and $s\in S_i\subseteq S$ such that $s'= h^{-1}sh$. By construction of $K_i$, $s$ acts elliptically on $T$, fixing (at least) a point of $K_i$ and therefore $s'$ acts elliptically, fixing a point of $h^{-1}\cdot L_i'\subseteq K_i'$.
\end{claimproof} 
\noindent The proof of Proposition~\ref{prop:expansion Gamma-tree} is now complete. \end{proof}

\noindent As a consequence of the previous two propositions, when proving the generalised twist conjecture for an Artin group $A$, it is sufficient to consider Artin generating sets $S$ and $U'$ admitting a tree $T$ which is both an $S$-tree and a $U'$-tree (up to equivariant isometry). This should be thought of as $S$ and $U'$ admitting a common visual splitting, which is what we now prove in Proposition \ref{prop:zigzag}. We will then have everything we need to prove Theorem \ref{thm:refrigidreduction}.

\begin{prop}\label{prop:zigzag}
    Given a one-ended Artin group $A$ with Artin generating sets $S$ and $U$ such that $R_S = R_U$, there exist equivariantly isometric $A$-trees $T_S$ and $T_{U'}$ with infinite cyclic edge stabilisers, such that $T_S$ is an $S$-tree and $T_{U'}$ is a $U'$-tree, for some Artin generating set $U'\sim_D U$. Moreover, if $A$ has at least two 1--chunks, $T_S$ contains an edge.
\end{prop}
\begin{proof}
    Take $M_S$ and $M_U$ as in Definition~\ref{def:crushed_tree}. One easily checks from Definition~\ref{def:crushed_tree} that if $A$ has at least two 1--chunks, then $M_S$ is minimal and is not just a point; these properties therefore hold for all $A$-trees in the corresponding deformation spaces.
    
    We claim the tree $M_S$ is an $S$-tree. Take the fundamental domain $K$ to be the lift of the graph of groups $\mathcal{M}_S$ described in Definition~\ref{def:crushed_tree} (which is possible since the underlying graph is a tree). Every standard 1--chunk $S$-parabolic subgroup occurs as a point stabiliser in $K$, and every $s\in S$ belongs to some standard 1--chunk parabolic subgroup, so every $s\in S$ fixes a point in $K$. Conversely, every simplex stabiliser in $K$ is either a standard 1--chunk $S$-parabolic subgroup, or generated by some $s\in S$ separating $\Gamma_S$, so is in particular a standard $S$-parabolic subgroup. Likewise, $M_U$ is a $U$-tree.

    Now, since $R_S = R_U$, Proposition~\ref{lem:crushedDependsOnReflections} gives that $M_S$ and $M_U$ are in the same deformation space $\crush D {R}$, where $R=R_S$. If $A$ has a single 1--chunk, then both $M_S$ and $M_U$ are a single point, and therefore there is an equivariant isometry between the underlying trees (which are single points with trivial $A$-actions). Otherwise $\crush D {R}$ is non-ascending by Corollary~\ref{lem:crushNonAscending}. Moreover, both $M_S$ and $M_U$ are surviving by Remark~\ref{rem:JS_survives}, so Theorem \ref{thm:GLSurvivingPath} produces a sequence of elementary collapses and expansions taking $M_U$ to a tree $T'$ which is equivariantly isometric to $M_S$. By applying Propositions~\ref{prop:collapse Gamma-tree} and~\ref{prop:expansion Gamma-tree}, we can realise $T'$ as a $U'$-tree, where $U'$ is as required by the statement.
\end{proof}

\begin{proof}[Proof of Theorem~\ref{thm:refrigidreduction}] Recall that we are given a one-ended Artin group $A$ and an Artin system $(A,S)$ such that, for every $X\subseteq S$ spanning a 1--chunk, $(A_X,X)$ satisfies the strong twist conjecture. Our goal is to prove that every Artin generating set $U$ for $A$ such that $R_S=R_U=R$ is generalised twist equivalent to $S$. We proceed by induction on the number of 1--chunks of $A$. The base case is immediate.

Before moving to the inductive step, let us clarify the inductive hypothesis. Let $\mathcal V$ be the collection of proper subsets $V\subsetneq S$ such that $A_V$ is one-ended and $V$ is of the form $V=X_1\cup\ldots\cup X_i$, where each $X_j$ spans a 1--chunk in $\Gamma_S$. By induction we can assume that, for every $V\in \mathcal V$, the Artin system $(A_V,V)$ satisfies the generalised strong twist conjecture.

Now, replacing $U$ by some $U'\sim_D U$ does not change the set of reflections, as argued in Definition~\ref{defn:twist_eq} and Lemma~\ref{lem:dehntwistpreservesgraph};  therefore, in light of Proposition~\ref{prop:zigzag}, we may assume that there are equivariantly isometric $A$-trees $(T_S, \Omega_S)$ and $(T_U, \Omega_U)$  with infinite cyclic edge stabilisers, such that $T_S$ is an $S$-tree and $T_U$ is a $U$-tree (we henceforth suppress the actions in our notation). Write $K_S$ and $K_U$ to be the corresponding combinatorial fundamental domains.

\par\medskip
    We now use the $S$-tree and $U$-tree structures to produce visual splittings of $A$ with respect to $S$ and $U$ such that both decompositions have the same factors, towards applying the inductive hypothesis to these factors. By Proposition~\ref{prop:zigzag}, $T_S$ has at least one edge, since $A$ has at least two 1--chunks; so fix an arbitrary edge $e$ of $K_S$, and let $\Stab{\Omega_S}{e} = \langle r \rangle$ for some $r \in S$. Take $f: T_S \rightarrow T_U$ to be an equivariant isometry. Up to conjugating $U$, we may assume that $f(e) \in \edge{K_U}$. By equivariance of $f$ it follows that $\Stab{\Omega_U}{f(e)} = \langle r \rangle$, and so either $r \in U$ or $r^{-1} \in U$. In fact, it must be that $r \in U$. If this was not the case, then since $r \in R_S = R_U$, there would be $u \in U$ conjugate to $r$. However $U$ would have one generator conjugate to the inverse of another, and this contradicts the existence of a homomorphism $A_U\to \Z$ mapping every generator to $1$.

    Write $\overline{T_S}$ (resp. $\overline{T_U}$) for the equivariant quotient $A$-tree obtained by collapsing every edge in $T_S$ (resp. $T_U$) not in the orbit of $e$ (resp. $f(e)$) to a point. Notice that these are in general not elementary collapses: in general $\overline{T_S}$ has more elliptic elements than $T_S$ and is in a different deformation space. By a routine diagram chase one sees that $f$ induces an equivariant isometry $\overline{f}\colon \overline{T_S}\to \overline{T_U}$.

    We write $S = S_1 \cup S_2$, where each $S_i$ is the set of generators fixing a vertex of $K_S$ on one side of $e$; this can be done since $K_S$ is a fundamental domain making $T_S$ an $S$-tree. Notice that $S_1 \cap S_2 = \{r\}$. We now recognise $\overline{T_S}$ as the Bass-Serre tree of the visual splitting corresponding to $S_1$ and $S_2$:

    \begin{claim}\label{claim:collapsedBassSerreTree}
        The $A$-tree $\overline{T_S}$ is the Bass-Serre tree for the splitting $A = A_{S_1} *_{\langle r \rangle} A_{S_2} $.
    \end{claim}
    \begin{claimproof}
        By equivariance of the quotient map $T_S\to \overline{T_S}$, we see that $\overline{T_S}$ has one orbit of edges. We claim it has two orbits of vertices. Suppose not, then, writing $e = \{x,y\}$ and $\{\overline x,\overline y\}$ for the image of $e$ in $\overline{T_S}$, there exists $g\in A$ such that $\overline x=g\overline y$. Since $\overline{T_S}$ is obtained from $T_S$ by collapsing all edges not in the orbit of $e$, there exists a path in $T_S$ from $x$ to $gy$, not passing through any edge in the orbit of $e$. However, this implies that $T_S / A$ is not simply connected, contradicting the fact that $T_S$ has a combinatorial fundamental domain by Proposition~\ref{lem:cfdTreeQuotient}.

        The above discussion identifies $\overline{T_S}$ with the Bass-Serre tree for the splitting $A = \Stab{A}{\overline{x}} *_{\Stab{A}{\overline{e}}} \Stab{A}{\overline{y}}$.
        Notice that $\Stab{A}{\overline{e}} = \langle r \rangle$, since the map $T_S\to \overline{T_S}$ is $A$-equivariant and injective on the interior of $e$.

        We now consider vertex stabilisers in $\overline{T_S}$. Without loss of generality, suppose $x$ is on the side of $e$ corresponding to $S_1$. Since $K_S$ contains only one edge from each orbit, all of the edges in $K_S$ on this side of $e$ are collapsed to $\overline{x}$. Each $s \in S_1$ fixes a vertex in $K_S$ on this side of $e$, and therefore fixes $\overline{x}$ by equivariance; this shows that $A_{S_1} \leq \Stab{A}{\overline{x}}$, and analogously $A_{S_2} \leq \Stab{A}{\overline{y}}$.

        Now let $B$ be the Bass-Serre tree for the splitting $A = A_{S_1} *_{\langle r \rangle} A_{S_2} $, and consider the $A$-equivariant, simplicial map $b\colon B\to \overline{T_S}$ mapping each coset $gA_{S_1}$ (resp. $gA_{S_2}$, $g\langle r\rangle$) to the corresponding coset $g\Stab{A}{\overline{x}}$ (resp. $g\Stab{A}{\overline{y}}$, $g\langle r\rangle$). This map is well-defined because $A_{S_1} \leq \Stab{A}{\overline{x}}$ and $A_{S_2} \leq \Stab{A}{\overline{y}}$, and it   preserves the incidence relation between edges and vertices as the latter corresponds to inclusion of cosets.
        
        Now $b$ is a bijection at the level of edges, as the latter correspond to $A$-cosets of $\langle r\rangle$ in both trees. At the level of vertices, $b$ is surjective because every vertex of $\overline{T_S}$ belongs to some edge, and it is injective because if two cosets of, say, $A_1$ map to the same coset of $\Stab{A}{\overline{x}}$, then the path between these cosets in $B$ maps to a backtracking path in $\overline{T_S}$, contradicting injectivity at the level of edges. Thus $b\colon B\to \overline{T_S}$ is an $A$-equivariant isometry, and in particular $\Stab{A}{\overline{x}} = A_{S_1}$ and $\Stab{A}{\overline{y}} = A_{S_2}$ as required.
    \end{claimproof}

    The arguments above apply verbatim with $U$ in place of $S$, substituting $f(e)$ for $e$ to obtain a splitting $A = A_{U_1} *_{\langle r \rangle} A_{U_1}$, where $U = U_1 \cup U_2$ and $U_1\cap U_2=\{r\}$. Notice that, by the equivariance of $\overline{f}$, the stabilisers of the endpoints of the image of $e$ in $\overline{T_S}$ and $\overline{T_U}$ are the same subgroups of $A$, so (up to exchanging $U_1$ and $U_2$), $A_{S_i} = A_{U_i}$ for $i \in \{1,2\}$. We denote these subgroups simply by $A_i$.
    
    The next two claims will allow us to use the inductive hypothesis:

    \begin{claim}\label{claim:S1_in_V} $S_1,S_2\in \mathcal V$.
    \end{claim}

    \begin{claimproof}[Proof of Claim~\ref{claim:S1_in_V}]
        We first notice that $S_1, S_2 \neq  S$, since otherwise $\overline{T_S}$, and hence $T_S$, would not be minimal, contradicting that it is an $A$-tree and hence minimal (Notation \ref{notation:action_on_trees}). Furthermore, each $X\subseteq S$ spanning a 1--chunk in $S$ lies in either $S_1$ or $S_2$. Indeed, $A_X$ fixes a point $p$ in $T_S$, since $T_S$ is in the same deformation space as $M_S$, and therefore has the same elliptic subgroups by Theorem~\ref{thm:forester}. This means that each $x\in X-\{r\}$ must fix a point $q\in K_S$. Notice that $q$ must belong to the same connected component of $T_S-e$ as $p$, since otherwise $x$ would fix the geodesic connecting $p$ and $q$ and in particular it would fix $e$.

        Finally, $A_i=A_{S_i}$ is one-ended since $\Gamma_{S_i}$ is connected. To see this, we shall prove that, for every $x\in S_i$ and any simple path  $\gamma\subseteq \Gamma_S$ connecting $x$ to $r$, we have that $\gamma\subseteq \Gamma_{S_i}$. Indeed, notice that $x$ and the first vertex of $\gamma$ after it, call it $z$, must belong to a common 1--chunk, say spanned by $X\subseteq S$. Moreover, by the above argument $X$ must belong to one of $S_1$ and $S_2$, and it must be that $X\subseteq S_i$ since $x\neq r$ belongs to $S_i$. Hence $z\in S_i$ as well. If $z=r$ we stop; otherwise we repeat this procedure, eventually showing that $\gamma\subseteq \Gamma_{S_i}$, as required.
    \end{claimproof}
    
    \begin{claim}\label{claim:reflections_in_chunk}
       For $i=1,2$ let $R_{S_i}^{A_i}$ and $R_{U_i}^{A_i}$ be the reflection sets for $S_i$ and $U_i$ in~$A_i$. Then $R_{S_i}^{A_i}=R_{U_i}^{A_i}$.
    \end{claim} 

    \begin{claimproof}[Proof of Claim~\ref{claim:reflections_in_chunk}]
    Let $s\in S_i$, and we have to find some $u'\in U_i$ which is conjugate to $s$ in $A_i$. Since $R_S=R_U$, there is $u\in U$ and $g\in A$ such that $s=gug^{-1}$. Moreover, notice that $\Gamma_{U_i}$ is a union of 1--chunks of $\Gamma_U$ (this follows as in Claim~\ref{claim:S1_in_V}, applied to $U$), so by e.g. \cite[Remark 3.6]{JonManSar_JSJSpl}, there exists a retraction $\rho\colon A\to A_i$ mapping every generator $v\in U$ to
    $$\rho(v)=\begin{cases}
        v&\mbox{ if }v\in U_i;\\
        r&\mbox{ if }v\not \in U_i.
    \end{cases}$$
    Since $\rho$ is the identity on $A_i$, we get that $s=\rho(s)=\rho(g)\rho(u)\rho(g)^{-1}$. Hence $s$ is conjugate to $\rho(u)\in U_i$ by $\rho(g)\in A_i$, as required.
    \end{claimproof}
    \noindent By induction, there exists $\psi: S_1 \rightarrow A_1$ which is a generalised twist equivalence between $S_1$ and $U_1$.

    \begin{claim}\label{claim:hatpsi}
        $\psi$ extends to a generalised twist equivalence $\hat\psi\colon S\to A$; moreover $S'\coloneq\hat\psi(S)=U_1\cup hS_2h^{-1}$, where $h\in A_1$ is such that $\psi(r)=hrh^{-1}$.
    \end{claim}

    \begin{claimproof}[Proof of Claim~\ref{claim:hatpsi}] Write $\psi$ as a sequence $\psi_l\circ\ldots\circ\psi_1$ of elementary twists, conjugations, and Dehn twists. By an inductive argument, it is enough to show that $\psi_1$ extends to an elementary twist, conjugation, or Dehn twist of $A$ with respect to~$S$. If $\psi_1$ is the conjugation by some $h_1\in A_1$, then let $\hat{\psi}_1: S \rightarrow A$ be the conjugation by $h_1$, and we have nothing to show. 
    
    Next, suppose $\psi_1$ is a Dehn twist of $(A_1,S_1)$. More precisely, let $v\in S_1$ and $B,C\subset S_1$ such that $A_1=A_{B\cup\{v\}}*_{\langle v\rangle} A_{C\cup \{v\}}$, and let $\psi_1$ be the Dehn twist of $B$ around $v$ by some $g\in A_{B\cup\{v\}}\cup A_{C\cup \{v\}}$ centralising $v$. Notice that $v$ separates $\Gamma_S$ as well, since $S_1\cap S_2=\{r\}$; furthermore, either $v=r$ or $S_2$ is in the same connected component of $\Gamma_{S}-\{v\}$ as $r$. If $\psi_1(r)=r$ (that is, if either $v=r$ or $r\in C$), we can define $\hat\psi_1\colon S\to A$ to be the identity on $S_2$, so that $\hat\psi_1$ is the Dehn Twist of $B$ around $v$ by $g\in A_{B\cup\{v\}}\cup A_{C\cup S_2\cup \{v\}}$. Otherwise $\psi_1(r)=grg^{-1}$, and if we define $\hat\psi_1$ by mapping every $s\in S_2$ to $gsg^{-1}$ we get the Dehn twist of $B\cup S_2$ around $v$ by $g\in A_{B\cup S_2\cup\{v\}}\cup A_{C\cup \{v\}}$. 
    
    A similar argument, with the required adjustments, proves that if $\psi_1$ is an elementary twist then it can be extended to an elementary twist $\hat\psi_1\colon S\to A$, thus proving the claim.\end{claimproof}

    \begin{claim}\label{claim:first_gte}
        $S'\sim_D U_1 \cup S_2$; moreover $r$ separates $\Gamma_{U_1\cup S_2}$.
    \end{claim}

    \begin{claimproof}
    Notice that $r, hrh^{-1} \in U_1$ are conjugate in $A_1$. By Lemma~\ref{lem:conjugatingStandardGenerators} there is a path of odd edges $\{e_1,\ldots, e_n\}$ in $\Gamma_{U_1}$ from $hrh^{-1}$ to $r$. For $i =1, \dots, n$ write $\Delta_i\in A_{U_1}=A_1$ for the Garside elements associated to the $e_i$, so that $\Delta_n\dots\Delta_1h$ centralises~$r$. 
    
    As in the proof of Claim~\ref{claim:twist_by_steps}, successively conjugating $hS_2h^{-1}$ by $\Delta_1$, $\Delta_2$, and so on, is a sequence of twists along separating edges (meaning that every $e_i$ is separating in the defining graph of $U_1 \cup \Delta_{i-1}\ldots\Delta_1 hS_2h^{-1}(\Delta_{i-1}\ldots\Delta_1)^{-1}$). Hence $$S'=U_1 \cup hS_2h^{-1} \sim U_1 \cup (\Delta_n\dots\Delta_1)hS_2h^{-1}(\Delta_n\dots\Delta_1)^{-1} =: S'',$$ and we note that $\Gamma_{S''}$ has $r$ as a separating vertex. Since $\Delta_n\dots\Delta_1h$ centralises $r$ and belongs to $A_{U_1}=A_1$, we can Dehn twist $S''$ by this element to obtain $U_1 \cup S_2$. We have now seen that $S_1 \cup S_2 \sim_D U_1 \cup S_2$. 
    
    For the moreover part, notice that $U_1 \cup S_2$ was obtained from $S''$ via a Dehn twist around $r$. Since Dehn twists preserve the defining graph by Lemma~\ref{lem:dehntwistpreservesgraph}, we see that $r$ is again separating in $\Gamma_{U_1 \cup S_2}$, as required.
    \end{claimproof}

    \noindent We are left to prove that $U_1 \cup S_2\sim_D U_1\cup U_2$. Again in view of Claims~\ref{claim:S1_in_V} and~\ref{claim:reflections_in_chunk}, the inductive hypothesis produces a generalised twist equivalence $\psi'$ between $S_2$ and $U_2$ in $A_2$. Since $r$ is separating in $U_1 \cup S_2$, we can then argue exactly as in Claim~\ref{claim:hatpsi} to extend $\psi'$ to a generalised twist equivalence $\hat\psi'\colon U_1 \cup S_2\to A$ with image $hU_1h^{-1}\cup U_2$, where $h\in A_2$ is such that $\psi'(r)=hrh^{-1}$. Finally, the arguments in the proof Claim~\ref{claim:first_gte} run verbatim to show that $\psi'(S)\sim_D U_1\cup U_2$, thus  completing the proof of Theorem~\ref{thm:refrigidreduction}. 
\end{proof}

\section{A genuine combination theorem using ribbons}\label{sec:strongtwist_for_some}
In Section~\ref{sec:red twist conj} we proved that an Artin system $(A,S)$ satisfies the generalised strong twist conjecture, provided that its 1--chunk parabolic subgroups satisfy the strong twist conjecture. In order to to improve the conclusion of Theorem~\ref{thm:refrigidreduction} to the genuine strong twist conjecture, we need to guarantee that, with respect to any Artin generating set that is twist-equivalent to $S$, Dehn twists around separating vertices can be written as a composition of elementary twists and conjugations. We prove that this condition is satisfied, whenever sufficiently many parabolic subgroups of $(A,S)$ enjoys the \emph{vertex ribbon property}. The latter describes the elements that conjugate standard generators via ribbons, which intuitively can be thought as minimal conjugating elements. The notion of a ribbon was introduced by Paris in~\cite{Pparabolics} in broader generality, and then studied in detail by Godelle (see Theorem~\ref{thm:properties_Godelle} below).

\begin{defn}[ribbons]
    \label{def:ribbon}
    Let $(A,S)$ be an Artin system and let $x,y\in S$. An element $g\in A$ such that $x=gyg^{-1}$ is a \emph{basic $(x,y)$-ribbon} (with respect to~$S$) if one of the following conditions hold:
    \begin{enumerate}
        \item the elements $x$ and $y$ are distinct, $m_{xy}\in\mathbb N_{\ge 3}$ is odd and $g=\Delta_{xy}$ (or its inverse);
        \item the elements $x$ and $y$ coincide and one of the following holds:
        \begin{enumerate}
            \item $g=x$ (or its inverse);
            \item there is $t\in S$ such that $m_{xt}\in\mathbb N_{\ge4}$ is even and $g=\Delta_{xt}$ (or its inverse);
            \item there is $t\in S$ such that $m_{xt} = 2$ and $g=t$ (or its inverse).
        \end{enumerate}
    \end{enumerate}
    An element $g=g_1\cdots g_n$ such that $x=gyg^{-1}$ is an \emph{$(x,y)$-ribbon} (with respect to $S$) if there exist $x_0,\dots, x_{n}\in S$ such that $x_0=x$, $x_{n}=y$ and, for every $i\in\{1,\dots,n\}$, $g_i$ is a basic $(x_{i-1},x_{i})$-ribbon. We denote the set of $(x,y)$-ribbons with respect to~$S$ by $\ribb{S}{x}{y}$.
    
    We say that the pair $(x,y)$ satisfies the \emph{vertex ribbon property} (in $(A,S)$) if, for every $g\in A$, if $x=gyg^{-1}$, then $g\in \ribb{S}{x}{y}$. We say that $(A,S)$ satisfies the \emph{vertex ribbon property} if every pair $(x,y)$ with $x,y\in S$ satisfies the ribbon property with respect to $S$. 
\end{defn}

\begin{rem}
    We will freely use that, if $g \in \ribb{S}{s}{t}$ and $h \in \ribb{S}{t}{r}$, then $gh \in \ribb{S}{s}{r}$ and $g^{-1} \in \ribb{S}{t}{s}$. 
\end{rem}

\begin{rem}\label{rem:basic_and_elm}
    The definition of ribbons that we are using is slightly different, albeit equivalent, to the one introduced by Paris and further studied by Godelle. In their original definition, ribbons are constructed as a concatenation of \emph{positive elementary ribbons} and their inverses~\cite[Definition 0.2]{godelle2003parabolic}, which correspond to conjugators of shortest length with respect to the partial order induced by the Garside structure every irreducible parabolic subgroup of spherical type is endowed with. It is straightforward from the definitions to check that every elementary ribbon can be written as a product of basic ribbons and vice versa. In particular, the (vertex) ribbon property can be defined using either basic ribbons or positive elementary ribbons, and the two properties coincide.
\end{rem}

\begin{rem}\label{rem:aa_ribbon}
    Let $(A,S)$ be an Artin system and $v\in S$. From Definition~\ref{def:ribbon}, it follows that, for every $(v,v)$-ribbon $h\in A$, there exists a sequence $(a_i,b_i)_{i=0}^n$ of pairs of elements in $S$, such that the following hold.
\begin{itemize}
    \item $a_0=b_n=v$.
    \item For every $i$, either $\{a_i,b_i\}$ span a dihedral Artin subgroup, and we set $m_i=m_{a_ib_i}$, or $a_i=b_i$, and with a little abuse of notation we set $m_i=1$.
    \item $a_{i+1}=a_i$ if $m_i$ is even, and $a_{i+1}=b_i$ if $m_i$ is odd.
    \item Set $t_i=b_i$ if $m_i\le 2$, while $t_i=\Delta_{a_ib_i}$ if $m_i\ge 3$.
    \item There exist $\varepsilon_i\in \{\pm 1\}$ such that $h= t_n^{\varepsilon_n}\ldots t_0^{\varepsilon_0}$.
\end{itemize}
In other words, the $\{a_i\}_{i=0}^n$ are the vertices of an odd loop $\gamma\subseteq \Gamma_S$, to which some even ``spikes" are glued, as in Figure~\ref{fig:loop_and_spikes}.
\end{rem}

\begin{figure}[htp]
    \centering
    \includegraphics[width=\textwidth, alt={Example of a loop with spikes associated to a ribbon conjugating a vertex to itself.}]{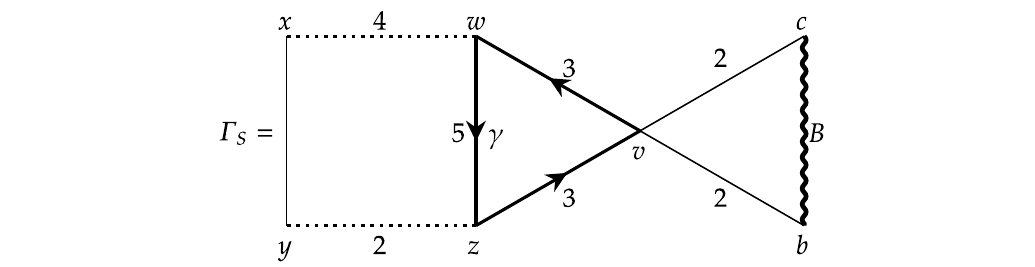}
    \caption{An example of a ``loop with spikes" associated to a $(v,v)$-ribbon, as in Remark~\ref{rem:aa_ribbon}. In the picture the ribbon is $h=\Delta_{vz}y^{-1}\Delta_{wz}\Delta_{wx}\Delta_{vw}^{-1}$, so we set $(a_0,b_0)=~(v,w)$, $(a_1,b_1)=~(w,x)$, $(a_2,b_2)=~(w,z)$, $(a_3,b_3)=~(z,y)$, and $(a_4,b_4)=~(z,v)$. The oriented path $\gamma$ corresponds to the collection of basic ribbons along odd edges, while the dotted ``spikes" are given by the basic ribbons along even edges. In the proof of Proposition~\ref{prop:aa_ribbon_are_twists}, we also consider a subset $B\subset S$ which $v$ separates from the rest of the vertices (here, the undulated line).}
    \label{fig:loop_and_spikes}
\end{figure}

\noindent In this paper, for the sake of explicitly describing Dehn twists, we only defined ribbons between vertices, which are are an instance of a more general definition of ribbons, encoding ``minimal'' group elements that conjugate standard parabolic subgroups: see e.g.~\cite[Definition 1.3]{godelle2007artin}. The vertex ribbon property, and a more general ``ribbon property", is conjecturally satisfied by all Artin groups, and was proven for several classes. We list some of them here, specialised to the cases we shall consider later:
\begin{thm}[{\cite{Pparabolics, godelle2002normalisateurs, godelle2003parabolic,godelle2007artin}}]\label{thm:properties_Godelle}
Let $(A,S)$ be an Artin system of either right-angled, spherical, or large type. Then $(A,S)$ satisfies the vertex ribbon property.
\end{thm}

\noindent The next proposition shows the relevance of the vertex ribbon property in this paper:

\begin{prop}\label{prop:aa_ribbon_are_twists}
    Let $(A,S)$ be an Artin system. Let $S=\{v\}\sqcup B\sqcup C$, where $m_{bc}=\infty$ for every $b\in B$ and $c\in C$, and let $h\in A_{B\cup\{v\}}\cup A_{C\cup\{v\}}$ centralise~$v$. Let $\delta$ be the Dehn twist of $B$ around $v$ by $h$. 
    If both $(A_{B\cup\{v\}},B\cup \{v\})$ and $(A_{C\cup\{v\}},C\cup \{v\})$ enjoy the vertex ribbon property, then $\delta(S)\sim S$.
\end{prop}

\begin{proof}
As in the proof of Lemma~\ref{lem:dehntwistpreservesgraph}, the Dehn Twist of $B$ by $h$ can be expressed as the Dehn twist of $C$ by $h^{-1}$, followed by the conjugation by $h$. Therefore it is enough to consider the case where $h\in A_{C\cup\{v\}}$. 

By the vertex ribbon property for $(v,v)$ in $C\cup\{v\}$, $h\in\ribb{C\cup\{v\}}{v}{v}$, so for every $i=0,\ldots, n$ let $(a_i,b_i)\in  C\cup\{v\}\times C\cup\{v\}$, $m_i$, $t_i$, $\varepsilon_i$, and $\gamma\subseteq \Gamma_{C\cup \{v\}}$ be as in Remark~\ref{rem:aa_ribbon}. 

Set $h_{-1}=1$ and, for every $i=0,\ldots, n$, let $h_i=t_i^{\varepsilon_i}\ldots t_0^{\varepsilon_0}$. Set $B_i=h_iBh_i^{-1}$, so that $B_i=t_i^{\varepsilon_i} B_{i-1} t_i^{-\varepsilon_i}$, and $S_i=C\cup\{v\}\cup B_i$, so that $S_n=\delta(S)$. We shall now prove by induction on $i$ that $S_i\sim S$, and furthermore that $a_{i+1}$ separates $\Gamma_{B_i}$ from the rest of $\Gamma_{S_i}$. The base case $i=-1$ holds vacuously. Now suppose the conclusion holds for $i-1$. There are four cases to consider, depending on $m_i$.

 If $m_i=1$ then $t_i= a_i$, so conjugating $B_{i-1}$ by $t_i$ (or its inverse) is an elementary twist around $\{a_i\}$. If $m_i=2$ then $t_i= b_i$, so conjugating $B_{i-1}$ by $t_i$ (or its inverse) is an elementary twist around $\{b_i\}$ (notice that $a_i\in \{b_i\}^\bot$ by construction). If $m_i\ge 4$ is even, then $t_i= \Delta_{a_ib_i}$, so conjugating $B_{i-1}$ by $t_i$ (or its inverse) is an elementary twist along the edge $\{a_i,b_i\}$. In each of these three cases $\Gamma_{S_i}\cong \Gamma_{S_{i-1}}$ and $a_{i+1}=a_i$, which still separates $\Gamma_{B_i}$ from the rest of $\Gamma_{S_i}$.

The last case to consider is when $m_i\ge 3$ is odd, so that $t_i= \Delta_{a_ib_i}$. The only difference with the even case is that $\Gamma_{S_i}$ is obtained from $\Gamma_{S_{i-1}}$ by ``sliding'' $\Gamma_{B_{i-1}\cup\{a_i\}}$ along the edge $\{a_i, a_{i+1}\}$ (the situation is similar to that in Figure~\ref{fig:slide}). Since $a_i$ was separating in $\Gamma_{S_{i-1}}$, $a_{i+1}$ now separates $\Gamma_{B_i}$ from the rest of $\Gamma_{S_i}$. This concludes the induction, and in turn the proof that $\delta(S)\sim S$. 
\end{proof}

We now study which procedures preserve the vertex ribbon property. We first observe that, to verify the vertex ribbon property, it is sufficient to understand centralisers of standard generators. We will freely use this fact in the sequel.

\begin{lemma}\label{lem:ss_rib_implies_st_rib}
   An Artin system $(A,S)$ satisfies the vertex ribbon property if and only if for all $s \in S$, $(s,s)$ satisfies the vertex ribbon property in $S$.
\end{lemma}
\begin{proof}
    The forward direction is obvious. For the reverse direction, suppose $s, t \in S$ and $t = gsg^{-1}$. Since $s$ and $t$ are conjugate, it follows from Lemma~\ref{lem:conjugatingStandardGenerators} that there is a path of odd edges between them, and so there is $h \in \ribb{S}{s}{t}$ (which is the product of the Garside elements associated to these edges). We see that $s = hth^{-1} = hgs(hg)^{-1}$, so $hg \in \ribb{S}{s}{s}$ by assumption. Hence $g \in h^{-1}\ribb{S}{s}{s} = \ribb{S}{t}{s}$ as required.
\end{proof}

\begin{prop}\label{prop:combination_ribbons}
    Let $(A,S)$ be an Artin system. Suppose that $A = A_B *_{A_D} A_C$ where $B,C \subseteq S$ and $D=B\cap C$. Suppose further that $(A_B,B)$ and $(A_C,C)$ satisfy the vertex ribbon property. Then $(A,S)$ satisfies the vertex ribbon property.
\end{prop}

\begin{proof} We consider the Bass--Serre tree $T$ for the splitting in the statement, following the construction of Example~\ref{ex:amalgProd}. We write $v_B$ and $v_C$ for the vertices corresponding to $A_B$ and $A_C$ respectively, and $e$ for the edge between those vertices, which is stabilised by $A_D$. 
    
By Lemma~\ref{lem:ss_rib_implies_st_rib}, it is enough to prove that, for every $s\in S$ and every $w\in A$ which commutes with $s$, $w\in \ribb{S}{s}{s}$. Without loss of generality assume that $s\in B$, so that $sv_B=v_B$. We write $T^s$ for the subtree of $T$ fixed by $s$, and notice that if $x \in T^s$ then $wx \in T^s$ as well. There are several cases to consider, according to the shape of $T^s$.

    \textbf{Case 1}. We suppose first that $s$ only fixes $v_B$. Then $wv_B = v_B$ so $w \in A_B$. Since $(A_B,B)$ satisfies the vertex ribbon property by hypothesis, we have that $w \in \ribb{B}{s}{s} \subseteq \ribb{S}{s}{s}$.
\par\medskip
    \textbf{Case 2a}. Next, suppose that $s$ fixes $e$, i.e. $s\in D$. We characterise edges in $T^s$:
    \begin{claim}\label{claim:rib}
        Every edge $ge$ in $T^s$ can be written as $fe$, where $f \in \ribb{S}{s}{t}$ for some $t \in D$.
    \end{claim}
    \begin{claimproof}[Proof of Claim \ref{claim:rib}]
        We proceed by induction on the distance between the midpoints of $ge$ and $e$. The base case is clear by taking $t = s$ and $f = 1$.

        For the inductive step, fix $fe \in T^s$ such that $f \in \ribb{S}{s}{t}$ for some $t\in D$. Suppose without loss of generality that $fv_C$ is the endpoint furthest from $e$, and consider an edge $ge$ sharing this endpoint in $T^s$, which we may write as $fhe$ for some $h \in A_C$. Notice that, since $sfhe = fhe$, we have that $h^{-1}th = h^{-1}f^{-1}sfh \in A_D$. By \cite[Theorem 1.1]{blufstein2023parabolic}, $h^{-1}\langle t \rangle h$ is a parabolic subgroup inside of $A_D$, so up to postmultiplying $h$ by an element of $A_D$ (which does not change the edge $fhe$), we may assume $h^{-1}th = r$ for some $r \in D$. By the assumption that $(A_C,C)$ satisfies the vertex ribbon property, we see that $h \in \ribb{C}{t}{r} \subseteq \ribb{S}{t}{r}$, and so $fh \in \ribb{S}{s}{r}$. This completes the proof of the claim, because $ge=fhe$.
    \end{claimproof}

    Now, take $w \in A_S$ such that $wsw^{-1} = s$, and therefore $we \in T^s$. By Claim~\ref{claim:rib}, we may write $we = fe$ where $f \in \ribb{S}{s}{t}$ for some $t \in D$. In turn $w = fh$, where $h \in A_D\le A_B$ is such that $hsh^{-1} = t$, and as such $h \in \ribb{B}{t}{s}\subseteq \ribb{S}{t}{s}$ by assumption that $(A_B,B)$ satisfies the vertex ribbon property. It follows that $w \in \ribb{S}{s}{s}$ as required.
\par\medskip

    \textbf{Case 2b}. Finally, suppose that $s$ fixes an edge, and we claim that we can reduce to Case 2a. Indeed, since $s$ fixes $v_B$, it must also fix some edge of the form $be$, where $b\in A_B$. Then $b\langle s\rangle b^{-1}$ is a parabolic subgroup of $A_D$, so by \cite{blufstein2023parabolic} there exists $d\in A_D$ and $s'\in D$ such that $b\langle s\rangle b^{-1}=d^{-1}\langle s'\rangle d$. By looking at the map $A_S\to \Z$ sending every generator to $1$ we see that $bsb^{-1}=d^{-1}s'd$. Now, since $db\in A_B$ conjugates $s\in B$ to $s'\in D\subseteq B$, the vertex ribbon property for $(A_B,B)$ yields that $db\in \ribb{B}{s'}{s}\subseteq \ribb{S}{s'}{s}$. Furthermore, since $wsw^{-1}=s$, the element $w'=db w(db)^{-1}$ commutes with $s'$, and the latter fixes $e$ as it belongs to $A_D$. Then $w'\in \ribb{S}{s'}{s'}$ by Case 2a, and therefore $w=(db)^{-1} w'db\in \ribb{S}{s}{s}$, as required.
\end{proof}

\begin{cor}\label{cor:ribbon_for_cliques_ribbon_for_all}
    Let $(A,S)$ be an Artin system. Suppose that, for any $Y \subseteq S$ spanning a clique in $\Gamma_S$, $(A_Y,Y)$ satisfies the vertex ribbon property. Then $(A,S)$ satisfies the vertex ribbon property. 
\end{cor}

\begin{proof}
    We proceed by induction on the cardinality of $S$, the base case $|S|=1$ being trivial. Now suppose the conclusion holds for every Artin system $(A',S')$ such that $|S'|<|S|$. If $\Gamma_S$ is a clique, there is nothing to prove. Otherwise let $u,v\in S$ be non-adjacent in $\Gamma_S$, and set $B=S-\{u\}$, $C=S-\{v\}$, and $D=B\cap C$. Every clique in $\Gamma_B$ or $\Gamma_C$ is also a clique in $\Gamma_S$; so by induction $(A_B,B)$ and $(A_C,C)$ satisfy the vertex ribbon property, and therefore so does $(A,S)$ by Proposition~\ref{prop:combination_ribbons}.
\end{proof}

\noindent The vertex ribbon property is also preserved under elementary twists, provided that it holds for cliques:
\begin{lemma}\label{lem:ribbons_preserved_under_twist}
    Let $(A,S)$ be an Artin system, and let $S'$ be obtained from $S$ by an elementary twist. Suppose further that, for any $Y \subseteq S$ spanning a clique in~$\Gamma_S$, $(A_Y,Y)$ satisfies the vertex ribbon property. Then for any $Y' \subseteq S'$ spanning a clique in $\Gamma_{S'}$, $(A_{Y'},Y')$ satisfies the vertex ribbon property. In particular $(A,S')$ satisfies the vertex ribbon property.
\end{lemma}

\begin{proof}
    By definition of an elementary twist, there exist $J,T,R\subseteq S$ such that, if we set $D=J\cup J^\bot$, $B=R\cup D$, $C=T\cup D$, then $S=B\cup C$ while $S'=B'\cup C$, where $B'=\Delta_J R\Delta_J^{-1} \cup D$. Notice that the conjugation by $\Delta_J$ preserves both $J$ and $J^{\bot}$, so $B'=\Delta_J B\Delta_J^{-1}$. Hence the conjugation by $\Delta_J$ defines an isomorphism $(A_B,B)\to (A_{B'},B')$. It follows that, for every $Y'\subseteq S'$ spanning a clique in $\Gamma_{B'}$ (resp. $\Gamma_C$), $(A_{Y'}, Y')$ has the vertex ribbon property by assumption. Furthermore any clique of $\Gamma_{S'}$ is contained in either $\Gamma_{B'}$ or $\Gamma_C$, as no $b'\in B'-D$ is adjacent to any $c\in C-D$ in $\Gamma_{S'}$. Finally, $(A,S')$ satisfies the vertex ribbon property by Corollary~\ref{cor:ribbon_for_cliques_ribbon_for_all}.
\end{proof}

\noindent We are now ready to prove that, if we assume the vertex ribbon property for cliques, then we can promote the generalised twist equivalence from Theorem~\ref{thm:refrigidreduction} to a genuine twist equivalence.

\begin{thm}\label{thm:strong+ribbon=strong}
    Let $(A,S)$ be an Artin system, with $A$ one-ended. Suppose that the following hold:
    \begin{itemize}
        \item For any $X\subseteq S$ spanning a 1--chunk in $\Gamma_S$, $(A_X,X)$ satisfies the strong twist conjecture;
        \item For any $Y\subseteq S$ spanning a clique in $\Gamma_S$, $(A_Y,Y)$ satisfies the vertex ribbon property.
    \end{itemize} 
    Then $(A,S)$ satisfies the strong twist conjecture.
\end{thm}

\begin{proof}
Theorem~\ref{thm:refrigidreduction} produces a sequence $S=S_0,\ldots,S_k=U$ of Artin generating sets such that, for every $i$, $S_{i+1}$ is obtained from $S_{i}$ by a conjugation, an elementary twist, or a Dehn Twist (with respect to $S_{i}$).

It is now enough to inductively prove that every $S_i$ is twist equivalent to $S$; we shall also prove that, for every $Y\subseteq S_i$ spanning a clique in $\Gamma_{S_i}$, $(A_Y,Y)$ satisfies the vertex ribbon property. There is nothing to prove for the base case $S_0=S$. Now assume that $S_i$ satisfies the inductive hypothesis. If $S_{i+1}$ is obtained from $S_i$ via a conjugation or an elementary twist, then $S_{i+1}\sim S$. Otherwise, there exist a separating vertex $v\in S_i$ and a decomposition $S_i=B\sqcup\{v\}\sqcup C$, with $m_{bc}=\infty$ whenever $b\in B$ and $c\in C$, such that $S_{i+1}$ is obtained from $S_i$ via a Dehn Twist of $B$ around $v$. Corollary~\ref{cor:ribbon_for_cliques_ribbon_for_all} and the inductive hypothesis then yield that $(A_{B\cup \{v\}},{B\cup \{v\}})$ and $(A_{C\cup \{v\}},{C\cup \{v\}})$ enjoy the vertex ribbon property, and therefore $S_{i+1}\sim S_{i}$ by Proposition~\ref{prop:aa_ribbon_are_twists}. Then Lemma~\ref{lem:ribbons_preserved_under_twist} implies that $(A_{Y'},Y')$ has the vertex ribbon property whenever $Y'\subseteq S_{i+1}$ spans a clique in $\Gamma_{S_{i+1}}$. This concludes the inductive step, and the proof of Theorem~\ref{thm:strong+ribbon=strong}.
\end{proof}

\section{New examples of the strong twist conjecture}
Here we gather several results in the literature and use them to produce several examples of Artin groups satisfying the strong twist conjecture.
\subsection{The right-angled case}

\begin{lemma}\label{lem:strong_RAAG}
    Let $(A,S)$ be a right-angled Artin system (i.e. for all $a,b\in S$, $m_{ab}\in \{2, \infty\}$), with $A$ one-ended. Then $(A,S)$ satisfies the strong twist conjecture.
\end{lemma}

\begin{proof}
Let $U$ be an Artin generating set for $A$ such that $R_S=R_U$. By work of Baudisch \cite{baudisch1981raagcharacterise} $U$ must be right-angled as well; in turn, a theorem of Droms \cite{droms1987isomorphisms} implies that two right-angled Artin systems on the same group are isomorphic. Hence let $\phi\colon A\to A$ be an automorphism such that $\phi(S)=U$, i.e. an isomorphism of Artin systems $(A,S)\to(A,U)$.

Now consider the abelianisation map $\mathrm{Ab}: A \rightarrow \mathbb{Z}^{|S|}$ mapping each $s\in S$ to an element of a basis. The kernel of $\mathrm{Ab}$, which is the commutator subgroup, is characteristic, so $\mathrm{Ab}$ induces a map $\mathrm{Ab}_*: \aut{A} \rightarrow \operatorname{GL}(|S|, \mathbb{Z})$ defined by mapping an automorphism $\psi\in \aut{A}$ to the map sending each coset $\ker(\mathrm{Ab}) x$ to $\ker(\mathrm{Ab}) \psi(x)$ (the observation that characteristic quotients induce maps between automorphism groups dates back to Baumslag \cite{Baumslag}). 

Since $\phi$ preserves $R_S$, it maps each $s\in S$ to a conjugate of some $s'\in S$. Hence $\mathrm{Ab}_*(\phi)$ is a permutation of the basis of $\mathbb{Z}^{|S|}$, and in turn this means that $\phi\in \ker(\mathrm{Ab}_*) \sigma$, where $\sigma\in\aut{A}$ is induced by a graph automorphism of $\Gamma_S$. It is now enough to notice that $\sigma$ fixes $S$ setwise, and, by \cite[Theorem 2.2]{laurence1995generating}, $\ker(\mathrm{Ab}_*)$ is generated by the partial conjugations with respect to $S$, which are exactly elementary twists in the sense of Definition \ref{defn:twist_eq}. This proves that $U=\phi(S)\sim S$. 
\end{proof}

\subsection{The large-type case}
We now prove the strong twist conjecture for several subclasses of large-type Artin systems, using results from~\cite{crisp2005automorphisms,BMV,Vaskou_LTFOI}.

\begin{lemma}\label{lem:strong_large}
Let $(A,S)$ be an Artin system of large type (i.e. for every $a,b\in S$, $m_{ab}\ge3$). Suppose that $\Gamma_S$ is connected and has no separating vertex and edges. Assume further that $(A,S)$ is either:
\begin{enumerate}
\item\label{item:dihedral} a dihedral Artin group;
\item\label{item:triangle-free} triangle-free (i.e. for every $a,b,c\in S$, $\max\{m_{ab},m_{bc},m_{ac}\}=\infty$);
\item\label{item:XXXL} of type XXXL (i.e. for every $a,b\in S$, $m_{ab}\ge6$);
\item\label{item:foi} free-of-infinity (i.e. for every $a,b\in S$, $m_{ab}<\infty$).
\end{enumerate}
Then $(A,S)$ satisfies the strong twist conjecture.
\end{lemma}

\begin{proof} 
Let $U$ be an Artin generating set for $A$ such that $R_S=R_U$, and remember that we want to show that $S\sim U$. By \cite[Corollary B]{MV_characterising_LT}, $\Gamma_S\sim\Gamma_U$; moreover, since spherical parabolic subgroups of large type Artin groups are supported on either vertices or edges, and since we are supposing that $\Gamma_S$ has no separating vertices nor edges, there is an isomorphism of labelled graphs $\Gamma_S\cong \Gamma_U$. As a consequence, there exists an automorphism $\phi\in \aut{A}$ such that $U=\phi(S)$. More precisely, $\phi$ is obtained as the composition of three isomorphisms
$$A\xrightarrow[]{f_S} A_{\Gamma_{S}}\xrightarrow[]{g}A_{\Gamma_{U}}\xrightarrow[]{f_{U}^{-1}}A,$$
where $f_S$ (resp. $f_{U}$) identifies $S$ with the vertices of $\Gamma_S$ (resp. $U$ with the vertices of $\Gamma_{U}$) and $g$ is induced by a graph isomorphism $\Gamma_S\cong \Gamma_{U}$.

\medskip 
We first address \eqref{item:dihedral}. Let $S=\{a,b\}$ and let $m=m_{ab}$. If $m$ is odd, then \cite[Theorem C]{Autdihedral} gives that $\out{A}\cong \Z/2\Z$ is generated by the global inversion, and therefore $\phi$ must be a conjugation since $R_S=R_{\phi(S)}$. 
    
Now assume that $m$ is even. Up to composing $\phi$ with a conjugation and the graph automorphism swapping $a$ and $b$, we can assume that $\phi(a)=a$. Then, by e.g. inspecting the proof of \cite[Lemma 2.9]{Jones_VF}, one gets that $\phi$ is either trivial or $\phi(b)=a^{-1}b^{-1}a^{-1}$. The latter cannot happen since $a^{-1}b^{-1}a^{-1}$ and $b$ have different images in the abelianisation, contradicting that $R_S = R_{\phi(S)}$.

\medskip We henceforth assume that $|S|\ge 3$, and we shall prove Cases~\eqref{item:triangle-free}-\eqref{item:XXXL}-\eqref{item:foi} at the same time. By~\cite[Theorem 1 and Theorem 2]{crisp2005automorphisms} for the triangle-free case, \cite[Theorem 9.6]{BMV} for the XXXL case, and \cite[Theorem A]{Vaskou_LTFOI} for the free-of-infinity case, $\aut{A}$ is generated by the following elements:
    \begin{enumerate}[label=(\roman*{})]
        \item\label{item_type1} inner automorphisms;
        \item\label{item_type2} graph automorphisms of $\Gamma_S$, which descend to automorphisms of $A\cong A_{\Gamma_S}$ and preserve $S$ setwise;
        \item the global inversion $\iota$ mapping every generator $s\in S$ to $s^{-1}$. 
    \end{enumerate}
    Note that here we are using that $\Gamma$ has no separating vertices, so in particular every vertex has valence at least two. This means we do not have to worry about the so called leaf-inversions, another class of automorphisms defined in e.g. \cite[page 1383]{crisp2005automorphisms}. We are also using the absence of separating edges to forbid the arising of \emph{edge twist isomorphisms}, again in the sense of e.g. \cite[1383]{crisp2005automorphisms}. 

    Since graph automorphisms of $\Gamma_S$ commute with the global inversion of $S$, and since inner automorphisms form a normal subgroup, we can write $\phi$ as $c_g\iota^\varepsilon \sigma$, where $\varepsilon\in \{0,1\}$, $c_g$ is the conjugation by some $g\in A$, and $\sigma$ is a graph automorphism of $\Gamma_S$. Thus $U=\phi(S)=g(\iota^\epsilon\circ\sigma)(S)g^{-1} =g\iota^\varepsilon(S)g^{-1}$, where we use that $\sigma$ permutes~$S$. If $\varepsilon=0$ then $U$ and $S$ are conjugate, hence twist equivalent. On the other hand, we cannot have that $\varepsilon =1$, as otherwise $U$ and $S$ would have different reflection sets because no generator of $S$ can be conjugate to the inverse of a generator of~$S$.
    \end{proof}

\subsection{The spherical-type case}
We finally turn to Artin systems of spherical type. Restricting to indecomposable systems (where there is no visual decomposition as a direct product) there are four infinite families. One family is the dihedral Artin groups, which we have already considered; in Figure~\ref{fig:spherical_artin} we list the Coxeter graphs of the Artin groups of type $A_n$, $B_n$ and $D_n$, which form the other three families. We shall now prove that spherical-type Artin system of type $A_n$, $B_n$ and $D_n$ also satisfy the strong twist conjecture. The only exception is $D_5$, whose outer automorphism group is not known.

\begin{figure}[htp]
    \centering
    \includegraphics[width=\linewidth, alt={Coxeter graphs of the aforementioned families of indecomposable spherical type Artin groups.}]{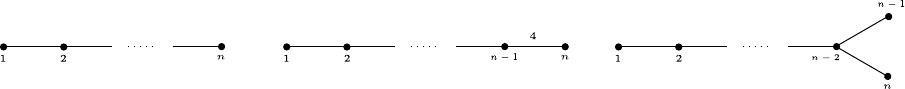}
    \caption{From left to right, the Coxeter graph of the spherical-type Artin group of type $A_n$, $B_n$ and $D_n$. The notation is different from the one we have been using throughout: here non-adjacent vertices in the Coxeter graph correspond to commuting generators, while unlabelled edges correspond with braid relations of length~$3$.}
    \label{fig:spherical_artin}
\end{figure}

\begin{lemma}\label{lem:strong_spherical}
    Let $(A,S)$ be a spherical-type Artin system of type $A_n$, with $n \geq 2$; $B_n$, with $n \geq 2$; or $D_n$, with $n \geq 4$ and $n \neq 5$. Then $(A,S)$ satisfies the strong twist conjecture.
\end{lemma}
\begin{proof}
    Write $W_S$ for the quotient Coxeter group, obtained as the quotient of $A$ with kernel $$K_S \coloneq  \nspan{s^2 \mid s \in S}=\langle gs^2g^{-1} \mid s\in S, g\in A\rangle.$$ The group $W_S$ is finite since $(A,S)$ is of spherical type. Suppose $U$ is an Artin generating set of $A$ with $R_S = R_U$, and define $W_U$ and $K_U$ as before. Since $R_S=R_U$, for every $u\in U$ there exist $s_u\in S$ and $h_u\in A$ such that $u=h_us_uh_u^{-1}$. Hence 
    $$K_U =\langle gu^2g^{-1} \mid u\in U, g\in A\rangle=\langle gs_u^2g^{-1} \mid u\in U, g\in A\rangle\le K_S.$$
    The same argument with $S$ and $U$ swapped shows that $K_S$ and $K_U$ are the same normal subgroup. Hence $W_S \cong W_U$, and in particular $W_U$ is finite, so that $(A,U)$ is also of spherical type by definition. By Paris' solution to the isomorphism problem within the class of spherical type Artin groups, it follows that $\Gamma_S \cong \Gamma_U$ \cite{Paris_isoprob_spherical}, so there is an automorphism $\psi: A \rightarrow A$ such that $\psi(S)  = U$.
    
     We will now show that $[\psi] \in \out{A}$ has a representative that simply permutes~$S$. Since such an automorphism fixes $S$ setwise, and replacing $S$ by a conjugate generating set is a twist equivalence, this will complete the proof.

     We now divide into cases based on the isomorphism type of $A$. Given $\phi_1, \phi_2 \in \aut{A}$, we will write $\phi_1 \sim \phi_2$ if $[\phi_1] = [\phi_2]$ in $\out{A}$. 

    \textbf{Case $A_n$:} In this case $\psi \sim \iota^i$ for $i \in \{0, 1\}$, where $\iota$ inverts each element of $S$ \cite{dyer1981automorphism}. Since $R_U = R_S$, it cannot be that $i = 1$, since then a generator in $S$ would be conjugate to the inverse of a generator in $S$, contradicting the existence of the map to $\mathbb{Z}$ sending each $s \in S$ to $1$.

    \textbf{Case $D_n$, $n \neq 5$:} In this case $\psi \sim \sigma\iota^i$, for $i \in \{0,1\}$, where $\iota$ is as before and $\sigma$ is a possibly trivial permutation of the generators \cite{soroko2021d4, castel2024endomorphisD}. Since $\sigma$ does not invert any generators, it must be that $i = 0$ for the same reason as in the previous case.

    \textbf{Case $B_n$:} This case is more complicated. Write $r_1, \dots, r_n$ for the standard generating set, where $m_{r_{n-1}r_n} = 4$, and set $\Delta_B = (r_1 \dots r_n)^n$ and $\delta = r_{n-1}\dots r_1 r_1 \dots r_{n-1}$. By inspecting the proof of \cite[Proposition~10]{charney2005automorphism} and the consequent remark, $\psi \sim T^k\iota^j\mu^i$, for $k \in \mathbb{Z}$ and $i, j \in \{0,1\}$, where $\iota$ is as above while $T$ and $\mu$ have the following forms (in each case $1 \leq \ell \leq n-1$):
    \[
    T\colon
    \begin{cases}
        r_\ell  \mapsto r_\ell\Delta_B\\
        r_n     \mapsto r_n\Delta_B^{-(n-1)}
    \end{cases}
    \text{ and }
    \mu\colon
    \begin{cases}
        r_\ell  \mapsto r_\ell^{-1}\\
        r_n     \mapsto \delta r_n
    \end{cases}
    \]
   (Via different methods, Paris and Soroko recovered the same result for the case $n\ge5$~\cite[Corollary 2.6]{paris2025endomorphisms}. The reader should notice that the generators extracted from the proof of \cite[Proposition~10]{charney2005automorphism} and those presented in the statement of \cite[Corollary 2.6]{paris2025endomorphisms} are the same, up to composition with inner automorphisms.)

    We consider the induced action of $\out{A}$ on the abelianisation, which is isomorphic to $\mathbb{Z}^2$ and generated by $\overline{r_1}$ (which is equal to $\overline{r_\ell}$ for $1 \leq \ell \leq n-1$) and $\overline{r_n}$. Since $R_S=R_U$, the automorphism $T^k\iota^j\mu^i$ has to preserve the generators setwise up to conjugacy, so the action on the abelianisation can be at most a permutation of coordinates. Now, by looking at the action on the $\overline{r_1}$ coordinate, since $|\overline{\Delta_B}| > 2$, we see that $k = 0$ and $i-j = 0$. Now, by looking at the action on the second coordinate, we see that $i = 0$, so $j = 0$, completing the proof. \end{proof}

We are finally ready to combine Theorem~\ref{thm:strong+ribbon=strong} with the result we gathered in this Section about the strong twist conjecture for several types of Artin groups:
\begin{cor}\label{cor:strong_twist_for_strong_chunks}
    Let $(A,S)$ be an Artin system, with $A$ one-ended. Suppose that, for every $X\subseteq S$ spanning a 1--chunk, $(A_X,X)$ is as in one of Lemmas~\ref{lem:strong_RAAG}-\ref{lem:strong_large}-\ref{lem:strong_spherical}. Then $(A,S)$ satisfies the strong twist conjecture.
\end{cor}

\begin{proof} Lemmas~\ref{lem:strong_RAAG}-\ref{lem:strong_large}-\ref{lem:strong_spherical} ensure that $(A_X,X)$ satisfies the strong twist conjecture whenever $X\subseteq S$ spans a 1--chunk. Now let $Y\subseteq S$ span a clique, which must be contained in some 1--chunk, say spanned by $X\subseteq S$. Notice that if $(A_X,X)$ is of large type (resp. of spherical type, right-angled) then so is $(A_Y,Y)$. Then $(A_Y,Y)$ enjoys the vertex ribbon property by Theorem \ref{thm:properties_Godelle}, and therefore Corollary~\ref{cor:strong_twist_for_strong_chunks} follows from Theorem~\ref{thm:strong+ribbon=strong}.
\end{proof}

\bibliography{biblio.bib}

@misc{an2022automorphismgroupsartingroup,
      title={The automorphism groups of {Artin} groups of edge-separated {CLTTF} graphs}, 
      author={Byung Hee An and Youngjin Cho},
      year={2022},
      eprint={2201.02502},
      archivePrefix={arXiv},
      primaryClass={math.GR},
      url={https://arxiv.org/abs/2201.02502}, 
}

@article{Dehn_isoprob,
 author = {Dehn, M.},
 title = {{\"U}ber unendliche diskontinuierliche {Gruppen}.},
 fjournal = {Mathematische Annalen},
 journal = {Math. Ann.},
 issn = {0025-5831},
 volume = {71},
 pages = {116--144},
 year = {1912},
 language = {German},
 doi = {10.1007/BF01456932},
 url = {https://eudml.org/doc/158521},
 zbMATH = {2630047},
 JFM = {42.0508.03}
}

@article{blufstein2023parabolic,
  title={Parabolic subgroups inside parabolic subgroups of {Artin} groups},
  author={Blufstein, Mart{\'\i}n and Paris, Luis},
  journal={Proceedings of the American Mathematical Society},
  volume={151},
  number={04},
  pages={1519--1526},
  year={2023}
}

@Misc{BMV,
 Author = {Blufstein, Mart{\'{\i}}n and Martin, Alexandre and Vaskou, Nicolas},
 Title = {Homomorphisms between {XL}-type {Artin} groups},
 Year = {2024},
 HowPublished = {Preprint, {arXiv}:2410.19091 [math.{GR}] (2024)},
 URL = {https://arxiv.org/abs/2410.19091},
 arXiv = {arXiv:2410.19091}
}

@Article{AutDihedral,
 Author = {Gilbert, N. D. and Howie, J. and Metaftsis, V. and Raptis, E.},
 Title = {Tree actions of automorphism groups},
 FJournal = {Journal of Group Theory},
 Journal = {J. Group Theory},
 ISSN = {1433-5883},
 Volume = {3},
 Number = {2},
 Pages = {213--223},
 Year = {2000},
 Language = {English},
 DOI = {10.1515/jgth.2000.017},
 URL = {semanticscholar.org/paper/d097854b7daec060147e10fa36e0a570c882d7ce},
 zbMATH = {1445115},
 Zbl = {0979.20027}
}

@Misc{Jones_VF,
 Author = {Jones, Oli},
 Title = {Type {VF} for outer automorphism groups of large-type {Artin} groups},
 Year = {2024},
 HowPublished = {Preprint, {arXiv}:2410.17129 [math.{GR}] (2024)},
 URL = {https://arxiv.org/abs/2410.17129},
 arXiv = {arXiv:2410.17129}
}

@article{crisp2005automorphisms,
  title={Automorphisms and abstract commensurators of 2--dimensional {Artin} groups},
  author={Crisp, John},
  journal={Geometry \& Topology},
  volume={9},
  number={3},
  pages={1381--1441},
  year={2005},
  publisher={Mathematical Sciences Publishers}
}

@phdthesis{VanDerLek,
  title={The homotopy type of complex hyperplane complements},
  author={Van der Lek, Harm},
  year={1983},
  school={Katholieke Universiteit te Nijmegen}
}

@Book{trees,
 Author = {Serre, Jean-Pierre},
 Title = {Trees. {Transl}. from the {French} by {John} {Stillwell}.},
 Edition = {Corrected 2nd printing of the 1980 original},
 FSeries = {Springer Monographs in Mathematics},
 Series = {Springer Monogr. Math.},
 ISSN = {1439-7382},
 ISBN = {3-540-44237-5},
 Year = {2003},
 Publisher = {Berlin: Springer},
 Language = {English},
 zbMATH = {1842475},
 Zbl = {1013.20001}
}

@Book{GuirardelLevitt,
 Author = {Guirardel, Vincent and Levitt, Gilbert},
 Title = {{JSJ} decompositions of groups},
 FSeries = {Ast{\'e}risque},
 Series = {Ast{\'e}risque},
 ISSN = {0303-1179},
 Volume = {395},
 ISBN = {978-2-85629-870-1},
 Year = {2017},
 Publisher = {Paris: Soci{\'e}t{\'e} Math{\'e}matique de France (SMF)},
 Language = {English},
 zbMATH = {6832257},
 Zbl = {1391.20002}
}

@Article{GL_def_trees,
 Author = {Guirardel, Vincent and Levitt, Gilbert},
 Title = {Deformation spaces of trees.},
 FJournal = {Groups, Geometry, and Dynamics},
 Journal = {Groups Geom. Dyn.},
 ISSN = {1661-7207},
 Volume = {1},
 Number = {2},
 Pages = {135--181},
 Year = {2007},
 Language = {English},
 DOI = {10.4171/GGD/8},
 zbMATH = {5166052},
 Zbl = {1134.20026}
}

@Article{Bass,
 Author = {Bass, Hyman},
 Title = {Covering theory for graphs of groups},
 FJournal = {Journal of Pure and Applied Algebra},
 Journal = {J. Pure Appl. Algebra},
 ISSN = {0022-4049},
 Volume = {89},
 Number = {1-2},
 Pages = {3--47},
 Year = {1993},
 Language = {English},
 DOI = {10.1016/0022-4049(93)90085-8},
 zbMATH = {433085},
 Zbl = {0805.57001}
}

@InCollection{pingpong,
 Author = {Mangahas, Johanna},
 Title = {The ping-pong lemma},
 BookTitle = {Office hours with a geometric group theorist},
 ISBN = {978-0-691-15866-2; 978-1-4008-8539-8},
 Pages = {85--105},
 Year = {2017},
 Publisher = {Princeton, NJ: Princeton University Press},
 Language = {English},
 DOI = {10.23943/princeton/9780691158662.003.0005},
 zbMATH = {7171761},
 Zbl = {1430.20020}
}

@InCollection{Muhlherr,
 Author = {M{\"u}hlherr, Bernhard},
 Title = {The isomorphism problem for {Coxeter} groups.},
 BookTitle = {The Coxeter legacy. Reflections and projections},
 ISBN = {0-8218-3722-2},
 Pages = {1--15},
 Year = {2006},
 Publisher = {Providence, RI: American Mathematical Society (AMS)},
 Language = {English},
 zbMATH = {5064000},
 Zbl = {1103.20031}
}

@article{Vaskou_LTFOI,
 author = {Vaskou, Nicolas},
 title = {Automorphisms of large-type free-of-infinity {Artin} groups},
 fjournal = {Geometriae Dedicata},
 journal = {Geom. Dedicata},
 issn = {0046-5755},
 volume = {219},
 number = {1},
 pages = {20},
 note = {Id/No 16},
 year = {2025},
 language = {English},
 doi = {10.1007/s10711-024-00951-x},
 zbMATH = {7972370}
}

@inproceedings {BMMNtwistconjecture,
    AUTHOR = {Brady, Noel and McCammond, Jonathan P. and M\"uhlherr, Bernhard and Neumann, Walter D.},
     TITLE = {Rigidity of {C}oxeter groups and {A}rtin groups},
 BOOKTITLE = {Proceedings of the {C}onference on {G}eometric and
              {C}ombinatorial {G}roup {T}heory, {P}art {I} ({H}aifa, 2000)},
   JOURNAL = {Geom. Dedicata},
  FJOURNAL = {Geometriae Dedicata},
    VOLUME = {94},
      YEAR = {2002},
     PAGES = {91--109},
      ISSN = {0046-5755,1572-9168},
   MRCLASS = {20F55 (20F36)},
  MRNUMBER = {1950875},
MRREVIEWER = {Vladimir\ N.\ Bezverkhni\u i},
       DOI = {10.1023/A:1020948811381},
       URL = {https://doi.org/10.1023/A:1020948811381}
}

@article{RT_Cox_not_twisteq,
 author = {Ratcliffe, John G. and Tschantz, Steven T.},
 title = {Chordal {Coxeter} groups.},
 fjournal = {Geometriae Dedicata},
 journal = {Geom. Dedicata},
 issn = {0046-5755},
 volume = {136},
 pages = {57--77},
 year = {2008},
 language = {English},
 doi = {10.1007/s10711-008-9274-9},
 zbMATH = {5382753},
 Zbl = {1175.20031}
}

@article{Paris_isoprob_spherical,
 author = {Paris, Luis},
 title = {{Artin} groups of spherical type up to isomorphism.},
 fjournal = {Journal of Algebra},
 journal = {J. Algebra},
 issn = {0021-8693},
 volume = {281},
 number = {2},
 pages = {666--678},
 year = {2004},
 language = {English},
 doi = {10.1016/j.jalgebra.2004.04.021},
 zbMATH = {2119117},
 Zbl = {1080.20033}
}

@article {Pparabolics,
    AUTHOR = {Paris, Luis},
     TITLE = {Parabolic subgroups of {A}rtin groups},
   JOURNAL = {J. Algebra},
  FJOURNAL = {Journal of Algebra},
    VOLUME = {196},
      YEAR = {1997},
    NUMBER = {2},
     PAGES = {369--399},
      ISSN = {0021-8693,1090-266X},
   MRCLASS = {20F36 (20F55)},
  MRNUMBER = {1475116},
MRREVIEWER = {Robert\ B.\ Howlett},
       DOI = {10.1006/jabr.1997.7098},
       URL = {https://doi.org/10.1006/jabr.1997.7098}
}

@article{forester,
author = {Forester, Max},
year = {2001},
month = {08},
pages = {219-267},
title = {Deformation and rigidity of simplicial group actions on trees},
volume = {6},
journal = {Geometry \& Topology},
doi = {10.2140/gt.2002.6.219}
}

@Article{garside,
 Author = {Garside, F. A.},
 Title = {The braid group and other groups},
 FJournal = {The Quarterly Journal of Mathematics. Oxford Second Series},
 Journal = {Q. J. Math., Oxf. II. Ser.},
 ISSN = {0033-5606},
 Volume = {20},
 Pages = {235--254},
 Year = {1969},
 Language = {English},
 DOI = {10.1093/qmath/20.1.235},
 zbMATH = {3308384},
 Zbl = {0194.03303}
}

@Misc{vaskou_isoproblem,
 Author = {Vaskou, Nicolas},
 Title = {The isomorphism problem for large-type {Artin} groups},
 Year = {2022},
 HowPublished = {Preprint, {arXiv}:2201.08329 [math.{GR}] (2022)},
 URL = {https://arxiv.org/abs/2201.08329},
 arXiv = {arXiv:2201.08329}
}

@article{godelle2007artin,
  title={{Artin}--{Tits} groups with {CAT(0)} {Deligne} complex},
  author={Godelle, Eddy},
  journal={Journal of Pure and Applied Algebra},
  volume={208},
  number={1},
  pages={39--52},
  year={2007},
  publisher={Elsevier}
}

@article{godelle2003parabolic,
  title={Parabolic subgroups of {Artin} groups of type {FC}},
  author={Godelle, Eddy},
  journal={Pacific journal of mathematics},
  volume={208},
  number={2},
  pages={243--254},
  year={2003},
  publisher={Mathematical Sciences Publishers}
}

@article{BrieskornSaito,
 author = {Brieskorn, Egbert and Saito, Kyoji},
 title = {{Artin}-{Gruppen} und {Coxeter}-{Gruppen}},
 fjournal = {Inventiones Mathematicae},
 journal = {Invent. Math.},
 issn = {0020-9910},
 volume = {17},
 pages = {245--271},
 year = {1972},
 language = {German},
 doi = {10.1007/BF01406235},
 url = {https://eudml.org/doc/142172},
 zbMATH = {3384300},
 Zbl = {0243.20037}
}

@article{godelle2002normalisateurs,
  title={Normalisateurs et groupes d'{Artin}-{Tits} de type sph\'erique},
  author={Godelle, Eddy},
  journal={Journal of Algebra},
    volume={269},
    number={1},
  year={2003}
}

@article {Baumslag,
    AUTHOR = {Baumslag, Gilbert},
     TITLE = {Automorphism groups of residually finite groups},
   JOURNAL = {J. London Math. Soc.},
  FJOURNAL = {The Journal of the London Mathematical Society},
    VOLUME = {38},
      YEAR = {1963},
     PAGES = {117--118},
      ISSN = {0024-6107,1469-7750},
   MRCLASS = {20.22},
  MRNUMBER = {146271},
MRREVIEWER = {Hanna\ Neumann},
       DOI = {10.1112/jlms/s1-38.1.117},
       URL = {https://doi.org/10.1112/jlms/s1-38.1.117},
}

@article{MV_characterising_LT,
 author = {Martin, Alexandre and Vaskou, Nicolas},
 title = {Characterising large-type {Artin} groups},
 fjournal = {Bulletin of the London Mathematical Society},
 journal = {Bull. Lond. Math. Soc.},
 issn = {0024-6093},
 volume = {56},
 number = {11},
 pages = {3346--3357},
 year = {2024},
 language = {English},
 doi = {10.1112/blms.13136},
 zbMATH = {7953414},
 Zbl = {1554.20076}
}

@article{droms1987isomorphisms,
  title={Isomorphisms of graph groups},
  author={Droms, Carl},
  journal={Proceedings of the American Mathematical Society},
  volume={100},
  number={3},
  pages={407--408},
  year={1987}
}

@article{baudisch1981raagcharacterise,
  title={Subgroups of semifree groups},
  author={Baudisch, Andreas},
  journal={Acta Mathematica Hungarica},
  volume={38},
  number={1-4},
  pages={19--28},
  year={1981},
  publisher={Akad{\'e}miai Kiad{\'o}, co-published with Springer Science+ Business Media BV~…}
}

@article{laurence1995generating,
  title={A generating set for the automorphism group of a graph group},
  author={Laurence, Michael R},
  journal={Journal of the London Mathematical Society},
  volume={52},
  number={2},
  pages={318--334},
  year={1995}
}

@article{JonManSar_JSJSpl,
    title={{JSJ} splittings for all {Artin} groups},
  author={Jones, Oli and Mangioni, Giorgio and Sartori, Giovanni},
  journal={ArXiv preprint arXiv:2506.17157},
  year={2025}
}

@article{charney2016problems,
  title={Problems related to {Artin} groups},
  author={Charney, Ruth},
  journal={American Institute of Mathematics},
  year={2016},
  publisher={Citeseer}
}

@article{dyer1981automorphism,
  title={The automorphism groups of the braid groups},
  author={Dyer, Joan L and Grossman, Edna K},
  journal={American Journal of Mathematics},
  volume={103},
  number={6},
  pages={1151--1169},
  year={1981},
  publisher={JSTOR}
}

@article{castel2024endomorphisD,
    AUTHOR = {Castel, Fabrice and Paris, Luis},
     TITLE = {Endomorphisms of {A}rtin groups of type {$D$}},
   JOURNAL = {Algebr. Geom. Topol.},
  FJOURNAL = {Algebraic \& Geometric Topology},
    VOLUME = {25},
      YEAR = {2025},
    NUMBER = {7},
     PAGES = {3975--4008},
      ISSN = {1472-2747,1472-2739},
   MRCLASS = {20F36 (57K20)},
  MRNUMBER = {4986961},
       DOI = {10.2140/agt.2025.25.3975},
       URL = {https://doi.org/10.2140/agt.2025.25.3975},
}

@article{soroko2021d4,
  title={{Artin} groups of types {$F_4$} and {$H_4$} are not commensurable with that of type {$D_4$}},
  author={Soroko, Ignat},
  journal={Topology and its Applications},
  volume={300},
  pages={107770},
  year={2021},
  publisher={Elsevier}
}

@article{paris2025endomorphisms,
title = {Endomorphisms of {Artin} groups of type ${B}_n$},
journal = {Journal of Algebra},
volume = {678},
pages = {831--861},
year = {2025},
issn = {0021-8693},
doi = {https://doi.org/10.1016/j.jalgebra.2025.03.017},
url = {https://www.sciencedirect.com/science/article/pii/S0021869325001474},
author = {Luis Paris and Ignat Soroko}
}

@article{charney2005automorphism,
  title={Automorphism Groups of some Affine and Finite Type {A}rtin Groups},
  author={Charney, Ruth and Crisp, John},
  journal={Mathematical Research Letters},
  volume={12},
  pages={321--333},
  year={2005}
}

@article{ratcliffe2013jsj,
    AUTHOR = {Ratcliffe, John G. and Tschantz, Steven T.},
     TITLE = {J{SJ} decompositions of {C}oxeter groups over {FA} subgroups},
   JOURNAL = {Topology Proc.},
  FJOURNAL = {Topology Proceedings},
    VOLUME = {42},
      YEAR = {2013},
     PAGES = {57--72},
      ISSN = {0146-4124,2331-1290},
   MRCLASS = {20F55 (20F65)},
  MRNUMBER = {2958988},
}
\bibliographystyle{alpha}
\end{document}